\documentclass[a4paper,11pt]{amsart}
\usepackage{comment}
\usepackage{hyperref}
\usepackage{amssymb, amsfonts, amsmath, amsthm,yhmath}
\usepackage{pdfpages}
\usepackage[utf8]{inputenc} 
\usepackage{tikz}

\usepackage{soul}

\newtheorem{thm}{Theorem}[section]
\newtheorem{cor}[thm]{Corollary}

\newtheorem{prop}[thm]{Proposition}

\newtheorem{rem}[thm]{Remark}
\newtheorem{conj}[thm]{Conjecture}

\newcommand {\pare}[1] {\left( {#1} \right)}
\newcommand {\lfrf}[1] {\lfloor {#1} \rfloor}
\newcommand{\bb}[1]{\mathbb{#1}}
\newcommand{\cal}[1]{\mathcal{#1}}

\newcommand{\interventier}[2]{\ensuremath{[\![#1,#2]\!]}}
\newcommand{\Lim}[2]{\lim\limits_{#1\rightarrow #2}}
\newcommand{\indic}{{\bf 1}}

\title{Iterated random walks in random scenery  (PAPAPA)}
\author{Nadine Guillotin-Plantard}
\address{Nadine Guillotin-Plantard, Universit\'e Claude Bernard Lyon 1, CNRS, Ecole Centrale de Lyon, INSA Lyon, Universit\'e Jean Monnet, ICJ UMR5208, 69622 Villeurbanne, France.}
\email{nadine.guillotin-plantard@univ-lyon1.fr}

\author{Françoise Pène}
\address{Françoise Pène, Univ Brest, Université de Brest,
UMR CNRS 6205, LMBA, 6 avenue Le Gorgeu, 29238 Brest cedex, France.}
\email{francoise.pene@univ-brest.fr}

\author{Frédérique Watbled}
\address{Frédérique Watbled, Université Bretagne Sud, UMR CNRS 6205, LMBA, F-56000 Vannes, France}
\email{frederique.watbled@univ-ubs.fr }

\begin{document}

\begin{abstract}
We establish a limit theorem for a new model of 3-dimensional random walk in an inhomogeneous lattice with random orientations. This model can be seen as a 3-dimensional version of the Matheron and de Marsily model~\cite{MdM}. This new model leads us naturally to the study of iterated random walk in random scenery, which is a new process that can be described as a random walk in random scenery evolving in a second random scenery. We use the french acronym PAPAPA for this process unprecedented in literature, and answer a question about its stochastic behaviour asked about twenty years ago by Stéphane Le Borgne.
\end{abstract}

\maketitle
\tableofcontents

\section{Model and main result}\label{secmod}

We consider a random walk in a stratified medium in $\mathbb Z^3$ with some random orientation of the horizontal lines. These orientations will be given by two independent sequences of Rademacher random variables 
$(\xi_z^{(1)})_{z\in\mathbb Z}$ and $(\xi_x^{(2)})_{x\in\mathbb Z}$. 
For all $x_0,y_0,z_0\in\mathbb  Z$,
\begin{itemize}
\item the lines parallel to $(0,0,1)$ are not oriented, the particle has the same probability to make a displacement to the top and to the bottom.
\item the lines parallel to $(1,0,0)$ of the plane
$\{(x,y,z_0) ; x,y \in \mathbb Z\}$ have orientation $(\xi^{(1)}_{z_0},0,0)$, i.e. the lines of the form $\{(x,y_0,z_0); x\in\bb Z\}$ have orientation 
$(\xi^{(1)}_{z_0},0,0)$.
\item the lines parallel to $(0,1,0)$ of the plane
$\{(x_0,y,z) ; y,z \in \mathbb Z\}$ have orientation $(0,\xi_{x_0}^{(2)},0)$, i.e. the lines of the form $\{(x_0,y,z_0); y\in\bb Z\}$ have orientation 
$(0,\xi_{x_0}^{(2)},0)$.
\end{itemize}

	\begin{figure}[h!] \caption{Example of orientations of  horizontal and vertical planes (in black and red) and at a point (in blue).}\label{GrapheWang0}
	\begin{tikzpicture}[scale=4]

	\draw[thick]  (-.2,1.8) -- (.3,1.8) -- (.13,1.75) ;
    \draw[thick]   (.3,1.8) -- (.13,1.85) ; 
     \draw[left]   (.3,1.71)  node {\huge $x$};
     \draw[thick]  (-.2,1.8) -- (-.2,2.3) -- (-.25,2.13) ;
    \draw[thick]   (-.2,2.3) -- (-.15,2.13) ; 
     \draw[left]   (-.2,2.3)  node {\huge $z$};
    \draw[thick]   (-.2,1.8)-- (.1,2.1)-- (0.05,1.95) ; 
        \draw[thick]   (.1,2.1)-- (-.08,2) ; 
     \draw[right]   (.1,2.1)  node {\huge $y$};

	\draw (.4,.4)--(0,0) -- (3.2,0) ;	
    \draw (.2,.1) -- (3,.1)--(2.83,.05) ;
    \draw  (3,.1)--(2.83,.15) ;
    \draw (.3,.2) -- (3.1,.2)--(2.93,.15) ;
    \draw  (3.1,.2)--(2.93,.25) ;
    \draw (.4,.3) -- (3.2,.3)--(3.03,.25) ;
    \draw  (3.2,.3)--(3.03,.35) ;
    
    	\draw[thick] (.4,.85)--(0,.45) -- (3.2,.45) ;	
    \draw[thick] (.37,.5)--(.2,.55) -- (3,.55) ;
    \draw[thick]  (.2,.55)--(.37,.60) ;
    \draw[thick] (.47,.6)--(.3,.65) -- (3.1,.65) ;
    \draw[thick]  (.3,.65)--(.47,.7) ;
    \draw[thick]  (.57,.7)--(.4,.75) -- (3.2,.75) ;
    \draw[thick]  (.4,.75)--(.57,.8) ;

	\draw (.4,1.3)--(0,.9) -- (3.2,.9) ;
    \draw (.37,.95)--(.2,1) -- (3,1) ;
    \draw  (.2,1)--(.37,1.05) ;
    \draw (.47,1.05)--(.3,1.10) -- (3.1,1.10) ;
    \draw  (.3,1.10)--(.47,1.15) ;
    \draw  (.57,1.15)--(.4,1.2) -- (3.2,1.2) ;
    \draw  (.4,1.2)--(.57,1.25) ;
	
	\draw[thick] (.5,1.85)--(0,1.35) -- (3.2,1.35) ;
    \draw[thick] (.2,1.45) -- (3,1.45) --(2.83,1.4) ;
    \draw[thick]  (3,1.45) --(2.83,1.50) ;
    \draw[thick]  (.3,1.55) -- (3.1,1.55) -- (2.93,1.5) ;
    \draw[thick]  (3.1,1.55)--(2.93,1.6) ;
    \draw[thick]  (.4,1.65) -- (3.2,1.65)--(3.03,1.6) ;
    \draw[thick]  (3.2,1.65)--(3.03,1.7) ;

	\draw[red,thick] (.45,1.6)--(.45,-.2)--(1.05,.4)--(1.05,2.2) --(.45,1.6) ;
	\draw[red,thick] (.47,0.02)--(.85,.38)--(.8,.25);
	\draw[red,thick] (.85,.38)--(.67,.28);
		\draw[red,thick] (.47,0.47)--(.85,.83)--(.8,.7);
	\draw[red,thick] (.85,.83)--(.67,.73);
		\draw[red,thick] (.47,0.92)--(.85,1.27)--(.8,1.15);
	\draw[red,thick] (.85,1.27)--(.67,1.18);
		\draw[red,thick] (.47,1.37)--(.85,1.72)--(.8,1.6);
	\draw[red,thick] (.85,1.72)--(.67,1.63);

\draw[red,thick] (1.35,1.6)--(1.35,-.2)--(1.95,.4)--(1.95,2.2) --(1.35,1.6) ;
	\draw[red,thick] (1.37,0.02)--(1.75,.38)--(1.7,.25);
	\draw[red,thick] (1.75,.38)--(1.57,.28);
		\draw[red,thick] (1.37,0.47)--(1.75,.83)--(1.7,.7);
	\draw[red,thick] (1.75,.83)--(1.57,.73);
		\draw[red,thick] (1.37,0.92)--(1.75,1.27)--(1.7,1.15);
	\draw[red,thick] (1.75,1.27)--(1.57,1.18);
		\draw[red,thick] (1.37,1.37)--(1.75,1.72)--(1.7,1.6);
	\draw[red,thick] (1.75,1.72)--(1.57,1.63);

	\draw[red,thick] (2.25,1.6)--(2.25,-.2)--(2.85,.4)--(2.85,2.2) --(2.25,1.6) ;
	\draw[red,thick] (2.32,.15)-- (2.27,0.02)--(2.65,.38);
	\draw[red,thick] (2.27,0.02)--(2.4,.05);
		\draw[red,thick] (2.32,.6)-- (2.27,0.47)--(2.65,.83);
	\draw[red,thick] (2.27,0.47)--(2.4,.5);
		\draw[red,thick] (2.32,1.05)-- (2.27,0.92)--(2.65,1.28);
	\draw[red,thick] (2.27,0.92)--(2.4,.95);
		\draw[red,thick] (2.32,1.55)-- (2.27,1.37)--(2.65,1.73);
	\draw[red,thick] (2.27,1.37)--(2.4,1.4);

		\filldraw[blue] (2.46,1.1) circle[radius=.8pt];
	\draw[blue,thick](2.46,1.1) -- (2.27,.92)--(2.32,1.05);
		\draw[blue,thick] (2.27,0.92)--(2.4,.95);
	\draw[blue,thick](2.46,1.1) -- (2.46,1.55)--(2.54,1.4);
		\draw[blue,thick](2.46,1.55)--(2.38,1.4);
			\draw[blue,thick](2.46,1.1) -- (2.46,.65)--(2.54,.8);
		\draw[blue,thick](2.46,.65)--(2.38,.8);
		\draw[blue,thick](2.46,1.1) -- (1.96,1.1)--(2.13,1.15);
				\draw[blue,thick] (1.96,1.1)--(2.13,1.05);

	\end{tikzpicture}
\end{figure}
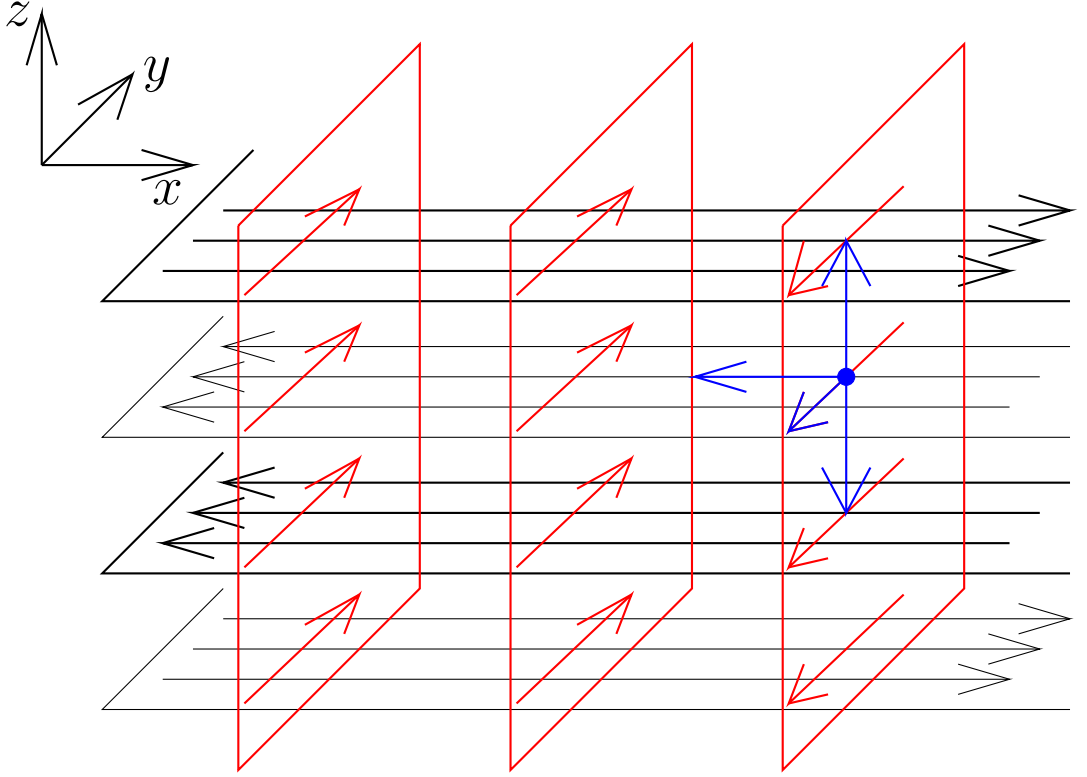

We consider a particle starting from the origin evolving in this medium making at each step a displacement in each of the three dimensions. Given the double scenery $(\xi^{(1)}_z,\xi^{(2)}_x)$, the motion is described by a Markov chain $\left(M_n=(X_n,Y_n,Z_n)\right)_{n\in\mathbb N}$ starting from the origin and with increments having the following distribution~:  
\[
M_n-M_{n-1}=
\left\{\begin{aligned}
&\left(\xi_{Z_{n-1}}^{(1)},\xi^{(2)}_{X_{n-1}},1\right)\textrm{ with probability }\frac{1}{2}\\
&\left(\xi_{Z_{n-1}}^{(1)},\xi^{(2)}_{X_{n-1}},-1\right)\textrm{ with probability }\frac{1}{2}
\end{aligned}\right.\, .\]
This process is given by 
\[
M_n=\left( X_n,Y_n,Z_n \right)\, ,
\]
with
\[
X_n:=\sum_{k=1}^{n}\xi^{(1)}_{Z_{k-1}},\quad Y_n:=\sum_{k=1}^{n}\xi^{(2)}_{X_{k-1}}\quad\mbox{and}\quad 
Z_n:=\sum_{k=1}^{n}\eta_k\, 
\]
where $(\eta_k)_{k\in\mathbb N^*}$ is a sequence of independent centered Rademacher random variables, independent of $(\xi_z^{(1)})_{z\in\mathbb Z}$ and $(\xi_x^{(2)})_{x\in\mathbb Z}$.
In this model, $(Z_n)_{n\in\mathbb N }$ is a simple random walk (PA=PA(1) in french, standing for \textit{Promenade Aléatoire}), $(X_n)_{n\in\mathbb N}$ is a random walk in random scenery (PAPA=PA(2) in french, standing for \textit{Promenade Aléatoire en Paysage Aléatoire}), and $(Y_n)_{n\in\mathbb N }$ is a random walk in random scenery in random scenery (PAPAPA=PA(3) in french). The question
of the asymptotic behaviour of PAPAPA was asked by Stéphane Le Borgne to the second  author about twenty years ago. We are now able to  answer it thanks to the developments made in the past two decades.  \\
In this paper, the convergence in distribution for a family of processes stands for weak
convergence in the space $\mathcal{D}=D([0,\infty[,\mathbb R)$ of processes with c\`adl\`ag trajectories equipped with the $J_1$-Skorohod topology on every compact subset.\\

{\bf Convergence in distribution of the PA.}
From Donsker's Theorem, the family of processes 
\begin{equation}\label{PA1}
\left(\mathcal P_n^{(1)} :=(Z_{\lfloor nt\rfloor}/n^{1/2})_t\right)_{n\in\mathbb N^*}
\end{equation}
converges in distribution to a standard Brownian motion $B$.\\
{\bf Convergence in distribution of the PAPA.}
It was proved in~\cite{KS-1979}  that 
\begin{equation}\label{PA2}
\left(\mathcal P_n^{(2)}:=(X_{\lfloor nt\rfloor}/n^{3/4})_t\right)_{n\in\mathbb N^*}
\end{equation}
converges in distribution
to the Kesten-Spitzer process $\Delta$ given by
\[\Delta(t)=\int_{\bb R} L^{(B)}(t,x)\,dW^{(1)}(x)\, ,\]
where $L^{(B)}$ is the local time of $B$ and $W^{(1)}$ is a
bilateral Brownian motion, independent of $B$.
Here, the bilateral Brownian motion $(W^{(1)}(x))_{x\in\mathbb R}$ is defined with real time and is simply a pair of independent Brownian motions
$(W^{(1)}_{+}, W^{(1)}_{-})$ so that the limiting process is namely defined by
\begin{equation}\label{4.5}
\Delta(t)=\int_{0}^{+ \infty}L^{(B)}(t,x)\, {\rm d}W^{(1)}_{+}(x)+\int_{0}^{+
\infty}L^{(B)}(t,-x)\, {\rm d}W^{(1)}_{-}(x).
\end{equation}
{\bf Joint convergence of the PA, PAPA and PAPAPA.} 
The existence and continuity of the local time $L^{(\Delta)}$ of $\Delta$ was established in~\cite{CGPPS-AOP-2014}. 
In this manuscript, we will prove the following limit theorem
that deals with the joint convergence of 
\[
\left(\mathcal P_n^{(3)}:=(Y_{\lfloor nt\rfloor}/n^{5/8})_t\right)_{n\in\mathbb N^*}\, ,\]
together with $\mathcal P_n^{(1)}$ and $\mathcal P_n^{(2)}$.
\begin{thm}\label{THM}
The sequence of 3-uples of processes 
\[\left(\mathcal P_n^{(2)},\mathcal P_n^{(3)},\mathcal P_n^{(1)}\right)_{n\in\mathbb N^*}=
\left(\left(X_{\lfloor nt\rfloor}/n^{3/4}\right)_t,\left(Y_{\lfloor nt\rfloor}/n^{5/8}\right)_t,\left(Z_{\lfloor nt\rfloor}/n^{1/2}\right)_t\right)_{n\in\mathbb N^*}\]
converges in distribution in $ \mathcal D([0,+\infty))^3$ to the 3-uple of processes $(\Delta,\Gamma,B)$,
where $B$ and $\Delta$ are as above and where
$\Gamma$ is given by
\[\Gamma(t):=\int_{\bb R} L^{(\Delta)}(t,x)\,dW^{(2)}(x)\, ,\]
where $W^{(2)}$ is a bilateral Brownian motion, independent of $(B,W^{(1)})$.
\end{thm}
Let us notice that, whereas the PAPA $\mathcal P^{(2)}$ is more diffusive than the random walk $\mathcal P^{(1)}$, the PAPAPA $\mathcal P^{(3)}$ is less diffusive than the PAPA $\mathcal P^{(2)}$.
This will be discussed at the beginning of Section~\ref{conjppt}.

We will establish this result under more general assumptions on the underlying random walk and on the random sceneries (see Theorem~\ref{THMtriple} in Section~\ref{PAPAPA}). 

It can also be interesting to consider the
random walk $(\widetilde M_n)_n$ to the closest neighbours in $\mathbb Z^3$
with the orientations $\xi^{(1)}$ and $\xi^{(2)}$ as described at the beginning of this section.
 Given the double scenery $(\xi^{(1)}_z,\xi^{(2)}_x)$, the motion is described by a Markov chain $\left(\widetilde M_n=\left(\widetilde X_n,\widetilde Y_n,\widetilde Z_n\right)\right)_{n\in\mathbb N}$ starting from the origin and with increments having the following distribution~:  
\[\widetilde M_n-\widetilde M_{n-1}=
\left\{\begin{aligned}
&(0,0,1)\textrm{ with probability }\frac{1}{4}\\
&(0,0,-1)\textrm{ with probability }\frac{1}{4}\\
&(\xi^{(1)}_{\widetilde Z_{n-1}},0,0)\textrm{ with probability }\frac{1}{4}\\
&(0,\xi^{(2)}_{\widetilde X_{n-1}},0)\textrm{ with probability }\frac{1}{4}
\end{aligned}\right.\]

	\begin{figure}[h!] \caption{A trajectory of $\widetilde M$.}\label{GrapheWang1}
	\begin{tikzpicture}[scale=4]
	\draw (.4,.4)--(0,0) -- (3.2,0) ;	
    \draw (.2,.1) -- (3,.1)--(2.83,.05) ;
    \draw  (3,.1)--(2.83,.15) ;
    \draw (.3,.2) -- (3.1,.2)--(2.93,.15) ;
    \draw  (3.1,.2)--(2.93,.25) ;
    \draw (.4,.3) -- (3.2,.3)--(3.03,.25) ;
    \draw  (3.2,.3)--(3.03,.35) ;
    
    	\draw[thick] (.4,.85)--(0,.45) -- (3.2,.45) ;	
    \draw[thick] (.37,.5)--(.2,.55) -- (3,.55) ;
    \draw[thick]  (.2,.55)--(.37,.60) ;
    \draw[thick] (.47,.6)--(.3,.65) -- (3.1,.65) ;
    \draw[thick]  (.3,.65)--(.47,.7) ;
    \draw[thick]  (.57,.7)--(.4,.75) -- (3.2,.75) ;
    \draw[thick]  (.4,.75)--(.57,.8) ;

	\draw (.4,1.3)--(0,.9) -- (3.2,.9) ;
    \draw (.37,.95)--(.2,1) -- (3,1) ;
    \draw  (.2,1)--(.37,1.05) ;
    \draw (.47,1.05)--(.3,1.10) -- (3.1,1.10) ;
    \draw  (.3,1.10)--(.47,1.15) ;
    \draw  (.57,1.15)--(.4,1.2) -- (3.2,1.2) ;
    \draw  (.4,1.2)--(.57,1.25) ;
	
	\draw[thick] (.5,1.85)--(0,1.35) -- (3.2,1.35) ;
    \draw[thick] (.2,1.45) -- (3,1.45) --(2.83,1.4) ;
    \draw[thick]  (3,1.45) --(2.83,1.50) ;
    \draw[thick]  (.3,1.55) -- (3.1,1.55) -- (2.93,1.5) ;
    \draw[thick]  (3.1,1.55)--(2.93,1.6) ;
    \draw[thick]  (.4,1.65) -- (3.2,1.65)--(3.03,1.6) ;
    \draw[thick]  (3.2,1.65)--(3.03,1.7) ;

	\draw[red,thick] (.45,1.6)--(.45,-.2)--(1.05,.4)--(1.05,2.2) --(.45,1.6) ;
	\draw[red,thick] (.47,0.02)--(.85,.38)--(.8,.25);
	\draw[red,thick] (.85,.38)--(.67,.28);
		\draw[red,thick] (.47,0.47)--(.85,.83)--(.8,.7);
	\draw[red,thick] (.85,.83)--(.67,.73);
		\draw[red,thick] (.47,0.92)--(.85,1.27)--(.8,1.15);
	\draw[red,thick] (.85,1.27)--(.67,1.18);
		\draw[red,thick] (.47,1.37)--(.85,1.72)--(.8,1.6);
	\draw[red,thick] (.85,1.72)--(.67,1.63);
	
	\draw[red,thick] (1.35,1.6)--(1.35,-.2)--(1.95,.4)--(1.95,2.2) --(1.35,1.6) ;
	\draw[red,thick] (1.37,0.02)--(1.75,.38)--(1.7,.25);
	\draw[red,thick] (1.75,.38)--(1.57,.28);
		\draw[red,thick] (1.37,0.47)--(1.75,.83)--(1.7,.7);
	\draw[red,thick] (1.75,.83)--(1.57,.73);
		\draw[red,thick] (1.37,0.92)--(1.75,1.27)--(1.7,1.15);
	\draw[red,thick] (1.75,1.27)--(1.57,1.18);
		\draw[red,thick] (1.37,1.37)--(1.75,1.72)--(1.7,1.6);
	\draw[red,thick] (1.75,1.72)--(1.57,1.63);

	\draw[red,thick] (2.25,1.6)--(2.25,-.2)--(2.85,.4)--(2.85,2.2) --(2.25,1.6) ;
	\draw[red,thick] (2.32,.15)-- (2.27,0.02)--(2.65,.38);
	\draw[red,thick] (2.27,0.02)--(2.4,.05);
		\draw[red,thick] (2.32,.6)-- (2.27,0.47)--(2.65,.83);
	\draw[red,thick] (2.27,0.47)--(2.4,.5);
		\draw[red,thick] (2.32,1.05)-- (2.27,0.92)--(2.65,1.28);
	\draw[red,thick] (2.27,0.92)--(2.4,.95);
		\draw[red,thick] (2.32,1.55)-- (2.27,1.37)--(2.65,1.73);
	\draw[red,thick] (2.27,1.37)--(2.4,1.4);

		\filldraw[blue] (2.46,1.1) circle[radius=.8pt];
			\draw[blue,thick](2.46,1.1) -- (2.46,.65)--(2.54,.8);
		\draw[blue,thick](2.46,.65)--(2.38,.8);
			\filldraw[blue] (2.46,.65) circle[radius=.8pt];
	\draw[blue,thick](2.46,.65) -- (2.36,.55)--(2.39,.68);
	\draw[blue,thick]  (2.36,.55)--(2.48,.57);
			\filldraw[blue] (2.36,.55) circle[radius=.8pt];
\draw[blue,thick](2.36,.55) -- (1.45,.55)--(1.62,.6);
\draw[blue,thick](1.45,.55)--(1.62,.5);
			\filldraw[blue] (1.45,.55) circle[radius=.8pt];
		\draw[blue,thick](1.45,.55) -- (1.45,1)--(1.37,.85);
	\draw[blue,thick]  (1.45,1)--(1.53,.85);
			\filldraw[blue] (1.45,1) circle[radius=.8pt];
				\draw[blue,thick](1.45,1) -- (1.55,1.1)--(1.52,.97);
	\draw[blue,thick]  (1.55,1.1)--(1.43,1.08);
			\filldraw[blue] (1.55,1.1) circle[radius=.8pt];
\draw[blue,thick](1.55,1.1) -- (1.55,1.55)--(1.47,1.35);
	\draw[blue,thick]  (1.55,1.55)--(1.63,1.35);
			\filldraw[blue] (1.55,1.55) circle[radius=.8pt];
\draw[blue,thick](1.55,1.55) -- (2.46,1.55)--(2.29,1.6);
\draw[blue,thick](2.46,1.55)--(2.29,1.5);
			\filldraw[blue] (2.46,1.55) circle[radius=.8pt];

	\end{tikzpicture}
\end{figure}
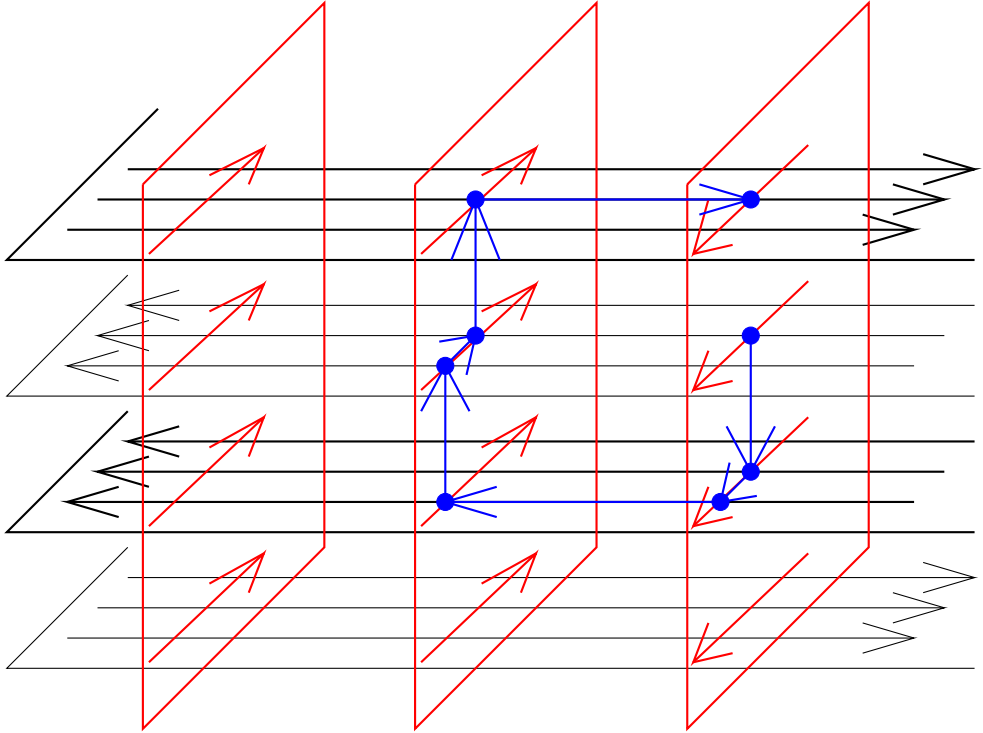

We notice that $\left(\widetilde Z_n\right)_{n\in\mathbb N}$ is a lazy symmetric random walk.
We observe that $\left(\widetilde X_n,\widetilde Z_n\right)_{n\in\mathbb N}$ corresponds to a lazy version of the two-dimensional Matheron-de Marsily model~\cite{MdM}. 
The convergence in distribution of $\left((\widetilde X_{\lfloor nt\rfloor}/n^{3/4}), (\widetilde Z_{\lfloor nt\rfloor}/n^{1/2})\right)_{n\in\mathbb N^*}$ follows\footnote{This is actually a "lazy version" of the walk considered in~\cite{GPLN1}. Let us write $(M'_n)_n$ for the process considered in~\cite{GPLN1} with $p=1/3$. Our process $\left(\widetilde Z_{n},\widetilde X_n\right)$ corresponds to $M'_{R_1+...+R_n}$ where $(R_i)_i$ is a sequence of i.i.d.
random variables with Bernoulli distribution with parameter $3/4$ and independent of the process $M'$. The announced FCLT follows from the FCLT of~\cite{GPLN1} via the classical time change argument~\cite[chapter 14]{Billingsley}.} from~\cite{GPLN1}. Furthermore a local limit theorem for $\left(\widetilde X_n,\widetilde Z_n\right)_{n\in\mathbb N}$ was proved in~\cite[Section 5]{CGPPS-AOP-2011}.

It should be possible to adapt the proof of Theorem~\ref{THM} to prove a limit theorem for $\left((\widetilde Y_{\lfloor nt\rfloor}/n^{5/8})_t\right)_{n\in\mathbb N^*}$ up to additional technicality (adaptation of the proofs of~\cite{Pene-EJP2021} with the use of~\cite[Section 5]{CGPPS-AOP-2011} applied with $p=1/4$ and $\mu_X$ the uniform distribution on $\{-1,0,1\}$).

{\bf Outline of the article.} 
Our main theorem is proved in Section~\ref{PAPAPA} (we establish therein a slightly more general result with weaker assumptions on the random walk and on the sceneries). 
The strategy of the proof consists in
using an induction argument that allows us:
\begin{itemize}
    \item[(I)] to recover the fact, established in~\cite{GPLN1}, that the convergence of $(\mathcal P_n^{(1)},\mathcal P_n^{(2)})$ follows from the convergence in distribution of $\mathcal P_n^{(1)}$ combined with
some multi-time versions of a local limit theorem for $\mathcal P_n^{(1)}$;
\item[(II)] to prove that the convergence of $(\mathcal P_n^{(1)},\mathcal P_n^{(2)},\mathcal P_n^{(3)})$ follows from the convergence in distribution of $(\mathcal P_n^{(1)},\mathcal P_n^{(2)})$ combined with
some multi-time versions of a local limit theorem for $\mathcal P_n^{(2)}$ (The required local limit estimates for
$\mathcal P_n^{(2)}$ have been established in~\cite{CGPPS-AOP-2014,Pene-EJP2021}).
\end{itemize}
To this end, we will use general results that are established in the previous Sections~\ref{joint} and \ref{secgene}.

We first state in Section~\ref{joint} a result of joint convergence in distribution for a process and its local time. This result is a joint version of a result of~\cite{DGP-EJP-2009}. Its short proof uses intermediate results obtained in the proof of~\cite{DGP-EJP-2009}. As a consequence
we obtain a corollary ensuring the joint convergence of a process together with the corresponding process in random scenery. This corollary will allow us to prove (I) and (II) above provided we are able to check some technical condition expressed in terms of the smoothness in the ${\bf L}^2$-norm of the local time of $\mathcal P_n^{(1)}$ (for case (I)) and then of $\mathcal P_n^{(2)}$ (for case (II)). 

In Section~\ref{secgene}, we use the results of Section~\ref{joint} to establish a joint limit theorem for stationary processes in random scenery, under assumptions based on (simple and double) local limit theorem estimates. These assumptions provide us a strategy to establish the technical hypotheses of the general result of Section~\ref{joint}.

In Section~\ref{PAPAPA}, we apply the results of Section~\ref{secgene} to prove a generalized version of Theorem~\ref{THM} (see Theorem~\ref{THMtriple}). We check the assumptions of the general result of
Section~\ref{secgene} using results recently established in~\cite{Pene-EJP2021} and by adapting a proof of~\cite{CGPPS-AOP-2014}.
We obtain Theorem~\ref{THM} as a consequence of this generalized theorem.

In Section~\ref{conjppt}, we discuss about the diffusivity of the successive $PA(p)$, and state a conjecture about the asymptotic behaviour of higher order $PA(p)$ for $p\ge 4$. This leads us naturally to the introduction of the expected limit objects
that correspond to continuous time processes $\Xi^{(p)}$ analogous to the discrete time processes $PA(p)$. Still in  Section~\ref{conjppt}, we study some properties of this limit processes 
(well-defined, self-similarity, stationary increments, H\"older continuous trajectories, existence of local time, existence of self-intersection local time, etc.).
In Section~\ref{othermodels}, we present other 3-dimensional models
with random orientations of lines, and discuss their relations with PAPAs or with open problems.

\section{A joint convergence result}\label{joint}
In this section, we prove that the functional argument used in~\cite{DGP-EJP-2009} to establish the convergence of random processes in random scenery can be adapted to prove joint convergence. This will be applied in the next section to prove the joint convergence stated in Theorem~\ref{THM}, i.e. of
\[\left(\mathcal P_n^{(2)},\mathcal P_n^{(3)},\mathcal P_n^{(1)}\right)_{n\in\mathbb N^*}\, ,\]
which corresponds, up to reordering this triple, to the joint convergence of the normalized PA, PAPA and PAPAPA.

To establish Theorem~\ref{THM}, we will apply the following proposition to  $A_n=\mathcal P_n^{(1)}$ and $X'_n=\mathcal P_n^{(2)}$
with $a_n=n^{3/4}$. 
In view of further generalizations of our results to iterated PAPA of any order $k$, $A_n$ can be thought as $(\mathcal P_n^{(1)},...,\mathcal P_n^{(k-2)})$ and $X'_n=\mathcal P_n^{(k-1)}$ (see Section~\ref{conjppt}). 

\begin{prop}\label{plocaltime}Let $\beta\in(1,2]$, $T>0$ and $M\in\mathbb N^*$. 
Let $(X_n)_n$ be an integer valued process and let $(A_n)_n$ be a
$(M-1)$-uple of processes in $\mathcal D([0,T])$. 
Assume that
\begin{itemize}
    \item[(a)] there exists a normalizing sequence $(a_n)_n$ diverging to infinity 
such that the family of processes $(A_n,X'_n:=X_{\lfloor n\cdot\rfloor}/a_n)_n$ converges in distribution to $(Z,\Delta)$ in $(\mathcal D([0,T]),J_1)^M$, 
where $(Z,\Delta)$ is a $M$-uple of processes in $\mathcal D([0,T])$ and where $\Delta$ admits a local time.
\item[(b)] Assumption {\bf (RW2)} of~\cite[Prop. 2.1]{DGP-EJP-2009} holds, that is for all positive real number $R$
 \[
 \lim_{h\rightarrow 0}
 \limsup_{n\rightarrow +\infty}\int_{[-R,R]}
 \mathbb E\left[\left|L_n(t,x) - L_n(t,h\lfloor x/h\rfloor)\right|^\beta\right]\, {\rm d}x = 0\, ,
 \]
 where we set
\[   L_n(t,x) := n^{-1}a_n \#\left\{k=1,...,\lfloor nt\rfloor\, :\, X_k=\lfloor a_nx\rfloor \right\}\, .
\]
\end{itemize}
Then, for every 
$t_1,...,t_m\in[0,T]$, the following convergence in distribution holds in the space $(\mathcal D([0,T]))^{M}\times ( {\mathbf L}^\beta(\mathbb R))^m$,
\begin{equation}\label{cv}
\left(A_n, X'_n, \left(L_n(t_j, \cdot)\right)_{j=1,...,m}\right) \Rightarrow \left(Z,\Delta, \left(L^{(\Delta)}(t_j, \cdot)\right)_{j=1,...,m}\right)\, ,
\end{equation}
as $n\rightarrow +\infty$.
\end{prop}
\begin{proof}Observe that we can express $L_n$ in terms of $X'_n$
as follows
\begin{equation}\label{defLn}
   L_n(t,x) = n^{-1}a_n \#\left\{k=1,...,\lfloor nt\rfloor\, :\, X'_n(k/n)=a_n^{-1}\lfloor a_nx\rfloor \right\}\, .
\end{equation}
The tightness of $(L_n(t_j,\cdot))_n$ seen as sequences of random variables with values in  ${\mathbf L}^\beta(\mathbb R)$ was proved in~\cite[Prop. 2.1]{DGP-EJP-2009}. This combined with Condition~(a) ensures the tightness of the left hand side of~\eqref{cv}. 
Let us consider $(Z''',\Delta''',W_1,...,W_m)$ the limit in distribution in  $\mathcal D([0,T])^{M}\times {\mathbf L}^\beta(\mathbb R)^m$ of a subsequence 
\begin{equation}\label{multi}
\left(A_{n_k}, X'_{n_k}, \left(L_{n_k}(t_j, \cdot)\right)_{j=1,...,m}\right)_k\, 
\end{equation}
of the left hand side of~\eqref{cv}.
By assumption~(a), $(Z''',\Delta''')$ has the same distribution as $(Z,\Delta)$. 
It follows from the Skorohod representation theorem that there exists, in a possibly enlarged probability space, a family of processes $(A''_{k},X''_{k},L''_{1,k},...,L''_{m,k})_k$ converging pointwise to $(Z'',\Delta'',W''_1,...,W''_m)$ in $\mathcal D([0,T])^{M}\times ({\mathbf L}^\beta(\mathbb R))^m$ such that $(Z'',\Delta'',W''_1,...,W''_m)$ has the same distribution as $(Z''',\Delta''',W_1,...,W_m)$ and such that, for every $k$,  $(A''_{k},X''_{k},L''_{1,k},...,L''_{m,k})_k$ has same distribution as~\eqref{multi}. Up to reducing if necessary the probability space, we assume without loss of generality (this event having probability 1) that, for every $k$, the process 
$X''_k$ takes its values in $\mathbb Z/a_{n_k}$ and has a set of discontinuity points contained in $\mathbb Z/n_k$
and, moreover, that
$L''_{j,k}$ coincides with
\[
L''_{j,k} (x)= (n_k)^{-1}a_{n_k} \#\left\{i=1,...,\lfloor n_kt_j\rfloor\, :\, X''_k(i/n_k)=\lfloor a_{n_k}x\rfloor/a_{n_k} \right\}=:L^{(X''_k)}(t_j,x)\, ,
\]
and that $\Delta''$ admits a continuous local time $L^{(\Delta'')}$.

As in~\cite[Proof of Prop. 2.1]{DGP-EJP-2009}, we observe that the convergence of $(X''_{k})_k$ to $\Delta''$ implies that the following convergence holds true for any $f$ in the set $D'$ of continuous functions with compact support~:
\begin{equation}\label{Wj0}
\forall j=1,...,m,\quad 
\lim_{k\rightarrow +\infty}\int_{\mathbb R}f(u)L^{(X''_{k})}(t_j,u)\, {\rm d}u=\int_{\mathbb R}f(u)L^{(\Delta'')}(t_j,u)\, {\rm d}u\, . \end{equation}

But the convergence of $(L''_{j,k})_k$ to $W''_j$ in ${\mathbf L}^\beta(\mathbb R)$ also implies that, for any $f\in D'$
\[
\lim_{k\rightarrow +\infty}\int_{\mathbb R}f(u)L^{(X''_{k})}(t_j,u)\, {\rm d}u=\lim_{k\rightarrow +\infty}\int_{\mathbb R}f(u)L''_{j,k}(u)\, {\rm d}u
=\int_{\mathbb R}f(u)W''_j(u)\, {\rm d}u\, .
\]
This combined with~\eqref{Wj0} ensures that, for any $f\in D'$, 
\[
\int_{\mathbb R}f(u)W''_j(u)\, {\rm d}u
=\int_{\mathbb R}f(u)L^{(\Delta'')}(t_j,u)\, {\rm d}u\, ,
\]
and so that
$W_j''=L^{(\Delta'')}(t_j,\cdot)$ in ${\mathbf L}^{\beta}(\mathbb R)$. 
Thus $\left(Z'',\Delta'',(W''_j)_j\right)$ 
has the same distribution as the right hand side of~\eqref{cv}. 
Therefore we have proved that every convergent subsequence (for the convergence in distribution) of the left hand side of~\eqref{cv} converges in distribution to the right hand side of~\eqref{cv}, as announced.
\end{proof}

\begin{cor}\label{COR}
Under the assumptions of Proposition~\ref{plocaltime}
Assume $(\xi_k)_{k\in\mathbb Z}$ is a sequence of i.i.d. random variables independent of $\left(A_n,X_n\right)_{n\in\mathbb N^*}$ belonging to the normal domain of attraction of a centered $\beta$-stable distribution $\cal S$, then we have the following convergence in distribution
\[
\left(A_n, X'_n, \left(n^{-1}a_n^{1-\frac 1\beta}\sum_{k=1}^{\lfloor n t_j\rfloor}\xi_{X_k}\right)_{j=1,...,m}\right) \Rightarrow \left(Z,\Delta, \left(\int_{\mathbb R}L^{(\Delta)}(t_j,x)\, {\rm d}U(x)\right)_{j=1,...,m}\right)\, 
\]
in  $\mathcal D([0,T])^{M}\times \mathbb R^m$, where $U$ is a bilateral $\beta$-stable Lévy process, independent of $(Z,\Delta)$ and such that $U(1)$ has distribution $\cal S$.
\end{cor}
\begin{proof}
It follows from Proposition~\ref{plocaltime} combined with the Skhorohod theorem that there exists a family of processes $(A''_{n},X''_{n},(L''_{j,n})_j)_n$ (defined on a same probability space) converging pointwise to $(Z'',\Delta'',(L^{(\Delta'')}(t_j,\cdot))_j)$ in $\mathcal D([0,T])^{M}\times ( {\mathbf L}^\beta(\mathbb R))^m$ such that $(Z'',\Delta'')$ has the same distribution as $(Z,\Delta)$ and such that, for every $n$,  $(A''_{n},X''_{n},(L''_{j,n})_j)_n$ has same distribution as
\[
\left(A_{n}, X'_{n}, \left(L_{n}(t_j, \cdot\right)_{j=1,...,m}\right)_n\, ,
\]
and
\[L''_{j,n}(x) = n^{-1}a_n N''_n( nt_j, a_nx)\, ,\]
with
\[
N''_n(u,y):=\#\left\{k=1,...,\lfloor u\rfloor\, :\, X''_n(k/n)=a_n^{-1}\lfloor y\rfloor\right\}\, .\]
Assume that $(\xi''_\ell)_{\ell\in\mathbb Z}$ is
a sequence of i.i.d. random variables with same distribution $\xi_1$ and 
independent of $(A''_{n},X''_{n},(L''_{j,n})_j)_n$. 
Then,
\begin{align*}
   n^{-1}a_n^{1-\frac 1\beta} \sum_{k=1}^{\lfloor n t_j\rfloor}\xi''_{a_nX''_n(k/n)}&= n^{-1}a_n^{1-\frac 1\beta}\sum_{\ell\in\mathbb Z}
       \xi''_\ell N''_n(nt_j,\ell)\\
     &=
      a_n^{-\frac 1\beta}\sum_{\ell\in\mathbb Z}
       \xi''_\ell \int_{\ell}^{\ell+1}L''_{j,n}(x/a_n)\, {\rm d}x\\
&=  
 a_n^{1-\frac 1\beta}
\sum_{\ell\in\mathbb Z}
       \xi''_\ell \int_{\ell/a_n}^{(\ell+1)/a_n}L''_{j,n}(z)\, {\rm d}z  \, .
\end{align*}

Therefore
\begin{equation}\label{TTT}
       n^{-1}a_n^{1-\frac 1\beta} \sum_{k=1}^{\lfloor n t_j\rfloor}\xi''_{a_nX''_n(k/n)}
       =
      \mu_{(a_n^{-1})}(L''_{j,n})\, ,
\end{equation}
where we write as in~\cite[(8)]{DGP-EJP-2009} (with $\gamma_{a_n^{-1}}=a_n^{-\frac 1\beta}$ as in~\cite[Section 3.1]{DGP-EJP-2009}):
\[
\mu_{(a_n^{-1})}(f)= a_n^{1-\frac 1\beta}\sum_{\ell\in\mathbb Z}\xi''_\ell
\int_{\ell/a_n}^{(\ell+1)/a_n}f(z)\, {\rm d}z\, .
\]

It follows from~\cite[Proposition 3.1]{DGP-EJP-2009} (applied to every scalar product of $\left(L''_{j,n}
\right)_{j=1,...,m}$) that, conditionally to $(A''_{n},X''_{n},(L''_{j,n})_j)_n)$, 
the sequence of random variables
\[\left(  \left(\mu_{(a_n^{-1})}(L''_{j,n})
\right)_{j=1,...,m}\right)_n
\]
converges in distribution in $\mathbb R^m$ to 
\[
\left(\int_{\mathbb R}L^{(\Delta'')}(t_j,x)\, {\rm d}U(x)\right)_{j=1,...,m}\, ,
\]
where $U$ is a bilateral $\beta$-stable Lévy process such that $U(1)$ has distribution $\cal S$. 
Due to~\eqref{TTT}, this implies that, conditionally to $(A''_{n},X''_{n},(L''_{j,n})_j)_n$, the sequence
\[
\left(A''_n ,X''_n ,\left(\mu_{(a_n^{-1})}(L''_{j,n} )
\right)_{j=1,...,m}\right)_{n\in\mathbb N^*}
\]
converges in distribution in $(\mathcal D([0,T]))^{M}\times ({\mathbf L}^\beta(\mathbb R))^m$  to
\[
\left(Z'' ,\Delta'' ,\left(\int_{\mathbb R}L^{(\Delta'' )}(t_j,x)\, {\rm d}U(x)\right)_{j=1,...,m}\right)\, .
\]
This ends the proof of the Corollary.
\end{proof}

\section{A joint convergence result using local limit theorems}\label{secgene}
In this section we consider as in the previous section 
a sequence $(A_n)_n$ of
$(M-1)$-uple of processes in $(\mathcal D([0,\infty]))^{M-1}$ converging jointly to some $(M-1)$-uple $Z$ of processes. 
We assume  furthermore that $(X_n)_{n\in\mathbb N}$ is
a $\mathbb{Z}$-valued random process 
with stationary increments. We assume that there exists
$d\in \mathbb N^{*}$ and $a\in\mathbb Z$ coprime such that
$\bb P(X_1\in a+d\mathbb Z)=1$ and that $d$ is the greatest integer satisfying this property.  
For example if $X_1$ is a Rademacher random variable as in the case of Theorem~\ref{THM}, we can take $a=1$ and $d=2$.  The forthcoming Proposition~\ref{PRO1} and Theorem~\ref{THMtriple} will allow more general situations. 
We consider also
a sequence $(\xi_y)_{y\in\mathbb Z}$  
of i.i.d. $\mathbb{R}$-valued centered random variables, independent of the processes $X_n$. 

We assume furthermore that the distribution of $\xi_0$ belongs to the normal domain of attraction of a stable distribution $\mathcal{S}_{\beta}(\sigma,\nu,0)$ with parameters 
$$\beta\in (1,2], \sigma>0 \ \mbox{and}\ \nu\in[-1,1],$$
its characteristic function being defined on $\mathbb R$ by
\begin{equation}\label{eq1.05}
 u \rightarrow \exp\left(-\sigma^\beta |u|^\beta \left(1-i\nu\tan\left(\frac{\pi\beta}{2}\right)\mbox{sgn}(u)\right)\right). 
\end{equation}
The following weak convergence holds:
\begin{equation}\label{cvl}
\left(n^{-\frac 1\beta}\sum_{k=1}^{\lfloor nt\rfloor}\xi_k\right)_{t\geq 0}
\xrightarrow[n\to\infty]{\cal L}\left( U(t)\right)_{t\geq 0}\, ,
\end{equation}
where $U$ is a $\beta$-stable Lévy process such that $U(0)=0$ and $U(1)$ has distribution $\mathcal{S}_{\beta}(\sigma,\nu,0)$.
In this setting, we are interested in the asymptotic behaviour, as $n\rightarrow +\infty$, of
$Y_n$ defined by
\[
Y_n:=\sum_{k=1}^n\xi_{X_{k-1}}\, .
\]
Our proof consists in applying the strategy of~\cite[Theorem 4.1]{DGP-EJP-2009} thanks to estimates based on the local limit 
theorem for the process $(X_n)_{n\in\mathbb N}$.\\
\begin{prop}\label{PRO1}
In addition to the previous assumptions, we assume that:
\begin{itemize}
\item[(i)] There exists $\alpha\in(0,1)$ s.t. the family of processes $(A_n,(X_{\lfloor nt\rfloor}/n^{\alpha})_t)_n$ converges in distribution in $(\mathcal D([0,T]))^M$ to a $M$-uple of processes $(Z,\Delta)$ with
$\Delta$ admitting a jointly continuous local time $L^{(\Delta)}=\left(L^{(\Delta)}(t,x)\right)_{t\geq 0, x\in\mathbb R}$.
\item[(ii)] $$
\sup_{n\in\mathbb N^*}\sup_{x\in\mathbb Z} n^\alpha\ \mathbb P(X_n=x)<\infty.$$
\item[(iii)] For any $x_1,x_2\in\bb R$ and $t_1,t_2>0$, 
\begin{multline}\label{cvxyn1}
n^{2\alpha}\bb P\left(X_{n_1}=\lfloor n^{\alpha}x_1\rfloor,X_{n_1+n_2}=\lfloor n^{\alpha}x_2\rfloor\right)\\
\sim  \indic_{\{n_1a-\lfloor n^\alpha x_1\rfloor\in d\mathbb Z\}}\indic_{\{(n_1+n_2)a-\lfloor n^\alpha x_2\rfloor\in d\mathbb Z\}}
d^2 \, p_{t_1,t_1+t_2}(x_1,x_2)\, ,
\end{multline}
as $n\to\infty$, $\frac{n_i}{n}\to t_i$
where $p_{T_1,T_2}$ is the density of the couple of random variables  $(\Delta_{T_1},\Delta_{T_2})$.
\item[(iv)] There exists $v\in(0,1-\alpha)$ and  $C>0$
 such that, for any integer $n\in\mathbb N^*$,
\[
\sup_{a_1,a_2\in\mathbb Z}\bb P(X_{k}=a_1 , X_{k+i}=a_2)\leq C(k i)^{-\alpha}
\textrm{ as soon as } k,i\in\interventier{\lfloor n^v\rfloor}{n}. 
\]
\item [(v)] for Lebesgue almost every $x\in\mathbb R$, 
$p_{s,t}$ is continuous at $(x,x)$.
\end{itemize}
Then, we have the following convergence in distribution
\[
\left(A_n,(X_{\lfloor nt\rfloor}/n^{\alpha})_t,n^{-1+\alpha(1-\frac 1\beta) }\sum_{k=1}^{\lfloor nt\rfloor}\xi_{X_{k-1}}\right)_{t\geq 0}
\Rightarrow\left( Z,\Delta,\int_{\mathbb R}L^{(\Delta)}(t,x)\, {\rm d}U(x)\right)_{t\geq 0}\, ,
\]
as $n\rightarrow +\infty$
in the space $(\mathcal D([0,+\infty]))^{M+1}$ endowed with the product of the usual $J_1$ topology on each compact interval $[0;T]$, and where $U$ is a bilateral $\beta$-stable Lévy process, independent of $\Delta$.
\end{prop}

\begin{rem}\label{notcoprime}
The condition that $a$ and $d$ are coprime can be relaxed as follows. 
If the assumptions of Proposition~\ref{PRO1} are satisfied with $X_1$ replaced
by $X'_1:=X_1/d_0$ for some positive integer $d_0$
(corresponding to the greatest integer such that $\mathbb P(X_1\in d_0\mathbb Z)=1$),
then, since $(\xi_{d_0 x})_{x\in\mathbb Z}$ has the same distribution as $(\xi_{x})_{x\in\mathbb Z}$, 
the process \[Y_n=\sum_{k=1}^n\xi_{X_{k-1}}
=\sum_{k=1}^n\xi_{d_0.X_{k-1}/d_0}
\]
has the same distribution as the process
\[
Y'_n:=\sum_{k=1}^n\xi_{X_{k-1}/d_0}\, .
\]
and so the conclusion of Proposition~\ref{PRO1} still holds true with the same process $U$ but with $\Delta$
the limit process of 
\[\left(\left(X'_{\lfloor nt\rfloor}/n^\alpha=X_{\lfloor nt\rfloor}/\left(d_0n^\alpha\right)\right)_t\right)_n\, .\]
\end{rem}
\begin{proof}[Proof of Proposition~\ref{PRO1}]
Our strategy consists in
applying Corollary~\ref{COR}
for the finite dimensional distributions of the last process, joint with the $M$ first processes, and by adapting the tightness result of~\cite{KS-1979}.\\
We start with the convergence of the finite dimensional distributions. 
By assumption, $(\xi_x)_{x\in\mathbb Z}$ satisfies the assumption of Corollary~\ref{COR}.
Assumption~(a) of Proposition~\ref{plocaltime} is ensured by Assumption~(i) of the present Proposition. 
In order to apply  Corollary~\ref{COR}, it remains to check Assumption~(b) of Proposition~\ref{plocaltime}. 
This will imply the convergence in distribution of
\begin{equation}\label{FDD}
\left(A_n,(X_{\lfloor nt\rfloor}/n^{\alpha})_t,\left(n^{-1+\alpha(1-\frac 1\beta) }\sum_{k=1}^{\lfloor nt_j\rfloor}\xi_{X_{k-1}}\right)_{j=1,...,m}\right)
\end{equation}
to
\begin{equation}\label{FDD0}
\left( Z,\Delta,\left(\int_{\mathbb R}L^{(\Delta)}(t_j,x)\, {\rm d}U(x)\right)_{j=1,...,m}\right)\, ,
\end{equation}
in $(\mathcal D([0,\infty[))^M\times\mathbb R^m$ as $n\rightarrow +\infty$. 
Due to~\cite{DGP-EJP-2009}, Assumption~(b) of Proposition~\ref{plocaltime} 
follows from Conditions~{\bf (RW2.a)} and~{\bf (RW2.b)} of~\cite[Remark after Proposition 2.1]{DGP-EJP-2009}.
To this end we will need some properties of the local time
$\mathcal N$ of the process $(X_n)_n$, which is given by
\[
\forall n\in\mathbb N^*,\quad \forall x\in\mathbb Z,\quad 
\mathcal N(n,x):=\#\left\{k=0,...,n-1\, :\, X_k=x\right\}\, ,
\]
i.e.
\begin{equation}\label{Nn}
\mathcal N(n,x):=\sum_{k=0}^{n-1}\indic_{\{X_k=x\}}\, .
\end{equation}
Taking $L_n(t,x)=n^{-(1-\alpha)} \mathcal N\left(\lfloor nt\rfloor ,\lfloor n^{\alpha}x\rfloor\right)$ in~\cite{DGP-EJP-2009}, {\bf (RW2.a)}  of~\cite{DGP-EJP-2009} can be deduced from the existence of $A>0$ such that
\begin{equation}\label{RW2a}
\forall n\in\mathbb N^*,\quad \sup_{x\in \mathbb Z} \bb E\left[\mathcal N(n,x)^2\right] \le A\, n^{2(1-\alpha)}\, .
\end{equation}
We set $Q(n,x):=\bb E\left[\mathcal N(n,x)^2\right]$, and control this quantity as follows
\begin{align*}
Q(n,x)&=\sum_{k,\ell=0}^{n-1}\bb P\left(X_k=X_\ell=x\right)
\\
&\le 2\sum_{k,m\ge 0\, :\, k+m\le n-1}\bb P\left(X_k=x,\ X_{k+m}-X_k=0\right)
\\
&\le 2n^v\sum_{k=0}^n\bb P\left(X_k=x\right)+2n^v\sum_{m=0}^n\bb P\left(X_m=0\right)+
2\sum_{k,m\ge n^v\, :\, k+m\le n-1}\bb P\left(X_k=X_{k+m}=x\right)\, ,
\end{align*}
where we used the fact that $X_{k+m}-X_k$ has the same distribution as $X_m$ since $(X_n)_n$ has stationary increments and $X_0=0$.
It thus follows from Assumptions~(ii) and~(iv)  of Proposition~\ref{PRO1} that
\begin{align*}
\sup_{x\in\mathbb Z}Q(n,x)
&\le \mathcal O\left(n^v\sum_{k=1}^n k^{-\alpha}
+\left(\sum_{k=\lfloor n^v\rfloor}^{n-1} k^{-\alpha}\right)^2\right)\\
&\le \mathcal O\left(n^v n^{1-\alpha}
+ \left(n^{1-\alpha}\right)^2\right)\\
&\le \mathcal O\left(n^{2(1-\alpha)}\right)
\, ,
\end{align*}
recalling that $v\le 1-\alpha$. 
This ends the proof of~\eqref{RW2a} and so of~{\bf (RW2.a)} of~\cite[Theorem 4.1]{DGP-EJP-2009}.\\
It remains to prove~{\bf (RW2.b)} of~\cite[Theorem 4.1]{DGP-EJP-2009}. We will prove that for all $t>0$, for every $x\in\bb R$,
\begin{equation}\label{RW2b}
\Lim{y}{x}\Lim{n}{\infty} \bb E\left[\left(L_n(t,x)-L_n(t,y)\right)^2\right]=0\, .
\end{equation}
Using the definition of $L_n$, this can be rewritten
\begin{equation}\label{RW2b1}
\forall t>0,\quad \forall x\in\mathbb R,\quad 
\Lim{y}{x}\Lim{n}{\infty} \bb E\left[n^{-2(1-\alpha)}\left(\mathcal N\left(\lfloor nt\rfloor ,\lfloor n^\alpha x\rfloor\right)-\mathcal N\left(\lfloor nt\rfloor ,\lfloor n^\alpha y\rfloor\right)\right)^2\right]=0\, .
\end{equation}
Let $t>0$ and $x\in\mathbb R$ be fixed.
For any $n\in\mathbb N^*$ and any $y\in\mathbb R$, we set
\[
E(n,x,y):=n^{-2(1-\alpha)}\bb E\left[\left(\mathcal N\left(\lfloor nt\rfloor ,\lfloor n^\alpha x\rfloor\right)-\mathcal N\left(\lfloor nt\rfloor ,\lfloor n^\alpha y\rfloor\right)\right)^2\right]\, .\]
It follows from~\eqref{Nn} that
\[
E(n,x,y)=n^{-2(1-\alpha)}\sum_{k_1,k_2=0}^{\lfloor nt\rfloor-1}\bb E\left[(\indic_{\{X_{k_1}=\lfloor n^\alpha x\rfloor\}}-\indic_{\{X_{k_1}=\lfloor n^\alpha y\rfloor\}})
(\indic_{\{X_{k_2}=\lfloor n^\alpha x\rfloor\}}-\indic_{\{X_{k_2}=\lfloor n^\alpha y\rfloor\}})\right]\, .\]
Whence
\begin{multline*}
E(n,x,y)=n^{-2(1-\alpha)}\sum_{k_1,k_2=0}^{\lfloor nt\rfloor-1} \left(
 \bb P\left(X_{k_1}=\lfloor n^\alpha x\rfloor=X_{k_2}\right)
+\bb P\left(X_{k_1}=\lfloor n^\alpha y\rfloor=X_{k_2}\right)\right.\\
-\left.\bb P\left(X_{k_1}=\lfloor n^\alpha x\rfloor,X_{k_2}=\lfloor n^\alpha y\rfloor\right)-\bb P\left(X_{k_1}=\lfloor n^\alpha y\rfloor,X_{k_2}=\lfloor n^\alpha x\rfloor\right)\right)\, . 
\end{multline*}
This can be rewritten as
\begin{equation}\label{Enxy}
E(n,x,y)= D(n,x,x)+D(n,y,y)-D(n,x,y)-D(n,y,x)\, ,
\end{equation}
setting
\begin{align}\nonumber
D(n,x_1,x_2)&:=n^{-2(1-\alpha)}\sum_{k_1,k_2=0}^{\lfloor nt\rfloor-1}\bb P\left(X_{k_1}=\lfloor n^{\alpha} x_1\rfloor,X_{k_2}=\lfloor n^{\alpha} x_2\rfloor\right)\, .
\end{align}
Let us prove that
\begin{align}\nonumber
D(n,x_1,x_2)&=o(1)+n^{-2(1-\alpha)}\sum_{(k_1,k_2)\in F_{\lfloor n t\rfloor}}
\left(\bb P\left(X_{k_1}=\lfloor n^{\alpha} x_1\rfloor,X_{k_2}=\lfloor n^{\alpha} x_2\rfloor\right)\right.\\
&\quad\quad\left.+\bb P\left(X_{k_1}=\lfloor n^{\alpha} x_2\rfloor,X_{k_2}=\lfloor n^{\alpha} x_1\rfloor\right)\right)\, ,\label{Dmab}
\end{align}
where we set
\[
F_m:=\left\{(k_1,k_2)\in\mathbb N^2\, :\, d\lfloor m^v\rfloor+1\le k_1\le k_1+ d\lfloor m^v\rfloor<  k_2\le k_1+\left\lfloor  \frac{m-k_1} {d}\right\rfloor  d,\ k_1\le d(\lfloor m/d\rfloor-1)\right\} \, .\]
This follows from the following estimates. Let $m=\lfloor nt \rfloor$.
First the case when $k_1$ or $k_2$ is smaller than $d\lfloor m^v\rfloor$ can be neglected since
\begin{align*}
\sum_{k_i\le d\lfloor m^v\rfloor, k_j=0,...,m-1 }\bb P\left(X_{k_i}=\lfloor n^{\alpha} x_i\rfloor,X_{k_j}=\lfloor n^{\alpha} x_j\rfloor\right)&\le 
2d\, m^v\sum_{k_j=0}^{m-1}\bb P\left(X_{k_j}=\lfloor n^{\alpha} x_j\rfloor\right)\\
&=\mathcal O\left(m^v\sum_{k=1}^{m-1}k^{-\alpha}\right)\\
&= o\left(m^{2(1-\alpha)}\right)\, ,
\end{align*}
using Assumption~(ii) and $v<1-\alpha$. Secondly the case when the gap
between $k_1$ and $k_2$ has length smaller than $d\lfloor m^v\rfloor$  is controlled as follows
\begin{align*}
\sum_{ k_i=0}^{m-1}\sum_{k_j=k_i}^{k_i+d\lfloor m^v\rfloor}\bb P\left(X_{k_i}=\lfloor n^{\alpha} x_i\rfloor,X_{k_j}=\lfloor n^{\alpha} x_j\rfloor\right)&\le 
d m^v\sum_{k_i=0}^{m-1}\bb P\left(X_{k_i}=\lfloor n^{\alpha} x_i\rfloor\right)\\
&=\mathcal O\left(m^v\sum_{k=0}^{m-1}k^{-\alpha}\right)= o\left(m^{2(1-\alpha)}\right)
\end{align*}
using Assumption~(ii) and $v<1-\alpha$. 

Finally, the sum over the $(k_1,k_2)$ such that $0\le k_i\le k_j$ and
such that 
$d(\lfloor m/d\rfloor-1)<k_i<m$ or $ k_i+\left\lfloor  \frac{m-k_i} { d}\right\rfloor  d <k_j<m$ is dominated by
 \[
 2d\sum_{k_j =0}^{m-1}\mathbb P\left(X_{k_j}=\lfloor n^{\alpha} x_j\rfloor\right)+
2d\sum_{k_i =0}^{m-1}\mathbb P\left(X_{k_i}=\lfloor n^{\alpha} x_i\rfloor\right)=o\left(m^{2(1-\alpha)}\right)\, .
 \]
This ends the proof of~\eqref{Dmab}, which can be rewritten as
\begin{equation}\label{Dnxy}
    D(n,x_1,x_2)=G(n,x_1,x_2)+G(n,x_2,x_1)+o(1)\, ,
\end{equation}
where
\begin{align}\nonumber
G(n,x_i,x_j):=n^{-2(1-\alpha)}\sum_{(k_1,k_2)\in F_{\lfloor n t\rfloor}}
\bb P\left(X_{k_1}=\lfloor n^{\alpha} x_i\rfloor,X_{k_2}=\lfloor n^{\alpha} x_j\rfloor\right)\, .
\end{align}
We will now use Assumption~(iii) of Proposition~\ref{PRO1}. 

The fact that $a$ and $d$ are coprime ensures that, for all  $x\in\mathbb Z$, for all $m\in\mathbb N$, there exists a unique integer $k$ such that
\[dm+1\leq k\leq d(m+1)
\textrm{ and }
ak-\lfloor n^{\alpha} x\rfloor\in d\bb Z.\]
With the change of variable $(k'_1,k'_2)=(k_1,k_2-k_1)$, the quantity  $G(n,x_i,x_j)$ can be rewritten as follows
\[
G(n,x_i,x_j)=n^{-2(1-\alpha)}\sum_{k'_1=d\lfloor \lfloor nt\rfloor^v\rfloor+1}^{d\left(\lfloor \frac {\lfloor nt\rfloor}d\rfloor-1\right)}\sum_{k'_2=d\lfloor \lfloor nt\rfloor^v\rfloor+1}^{\lfloor (\lfloor nt\rfloor-k'_1)/d\rfloor d} 
\bb P\left(X_{k'_1}=\lfloor n^{\alpha} x_i\rfloor,X_{k'_1+k'_2}=\lfloor n^{\alpha} x_j\rfloor\right)\, .
\]
Decomposing this sum in blocks of size $d$, this becomes
\begin{multline*}
G(n,x_i,x_j)=\\n^{-2(1-\alpha)}\sum_{m_1=\lfloor \lfloor nt\rfloor^v\rfloor}^{\lfloor \frac {\lfloor nt\rfloor}d\rfloor-2}\sum_{m_2=\lfloor \lfloor nt\rfloor^v\rfloor}^{\lfloor (\lfloor nt\rfloor-k_1(m_1))/d\rfloor -1}
\bb P\left(X_{k_1(m_1)}=\lfloor n^{\alpha} x_i\rfloor,X_{k_1(m_1)+k_2(m_2)}=\lfloor n^{\alpha} x_j\rfloor\right)
\, 
\end{multline*}
where $k_1(m_1)$ stands for the unique integer such that
\[dm_1+1\leq k_1(m_1)\leq d(m_1+1)
\textrm{ and }
ak_1(m_1)-\lfloor n^{\alpha} x_j\rfloor\in d\bb Z\]
and 
$k_2(m_2)$ stands for the unique integer such that
\[dm_2+1\leq k_2(m_2)\leq d(m_2+1)
\textrm{ and }
a(k_1(m_1)+k_2(m_2))-\lfloor n^{\alpha} x_j\rfloor\in d\bb Z.\]
Setting
\[F_n(u_1,u_2,x_i,x_j)=n^{2\alpha}\bb P\left(X_{k_1(\lfrf{u_1})}=\lfloor n^{\alpha} x_i\rfloor,X_{k_1(\lfrf{u_1})+k_2(\lfrf{u_2})}=\lfloor n^{\alpha} x_j\rfloor\right),\]
this can be rewritten under an integral form as follows
\begin{align*}
G(n,x_i,x_j)
&=n^{-2}\int_{\lfloor {\lfloor nt\rfloor}^v\rfloor}^{\lfloor \frac {\lfloor nt\rfloor}d\rfloor-1}
\int_{\lfloor {\lfloor nt\rfloor}^v\rfloor}
^{\lfloor\frac{\lfloor nt\rfloor-k_1(\lfloor u_1\rfloor)}{d}\rfloor}
F_n(u_1,u_2,x_i,x_j)
\, {\rm d}u_1\, {\rm d}u_2\\
&=\int_{[0, t/d]^2}\indic_{I_{n,t}}(u_1,u_2)
F_n(nu_1,nu_2,x_i,x_j)\, {\rm d}u_1\, {\rm d}u_2
\, ,
\end{align*} where $I_{n,t}$ is the set of $(u_1,u_2)\in[0,1]^2$ such that
$\lfloor \lfloor nt\rfloor^v\rfloor/n\le u_1\le (\lfloor \frac {\lfloor nt\rfloor} d\rfloor-1)/n$
and $\lfloor \lfloor nt\rfloor^v\rfloor/n\le u_2\le \lfloor (\lfloor nt\rfloor -k_1(\lfloor nu_1\rfloor))/d\rfloor/n.$
It follows from the dominated convergence theorem (the pointwise convergence comes from Assumption~(iii) of Proposition~\ref{PRO1} and the domination by $(2C/(u_1u_2)^{\alpha}$
from Assumption~(iv) applied with $\lfloor nt\rfloor$ instead of $n$)  that
\begin{align*}
\lim_{n\rightarrow +\infty}G(n,x_i,x_j)&=\int_{ u_1,u_2\ge 0\, :\,  u_1+u_2\le t/d} d^2 p_{du_1,d(u_1+u_2)}(x_i,x_j)\, {\rm d}u_1\, {\rm d}u_2\\
&=\int_{v_1,v_2\ge 0\, :\, v_1+ v_2\le t} p_{v_1,v_1+v_2}(x_i,x_j)\, {\rm d}v_1\, {\rm d}v_2\\
&=\int_{0\le t_1\le t_2\le t}p_{t_1,t_2}(x_i,x_j)\, {\rm d}t_1\, {\rm d}t_2
\, .
\end{align*}
This combined with~\eqref{Dnxy} ensures that
\[
\lim_{n\rightarrow +\infty}D(n,x_1,x_2) = \int_{0\le t_1\le t_2\le t}(p_{t_1,t_2}(x_1,x_2)+p_{t_1,t_2}(x_2,x_1))\, {\rm d}t_1\, {\rm d}t_2\, .
\]
Using now~\eqref{Enxy}, we have proved that the quantity $E(n,x,y)$ appearing in Condition~{\bf (RW2.b)} satisfies
\begin{align*}
\lim_{n\rightarrow +\infty}E(n,x,y)&=2\int_{0\le t_1\le t_2\le t}\left(p_{t_1,t_2}(x,x)+p_{t_1,t_2}(y,y)-p_{t_1,t_2}(x,y)-p_{t_1,t_2}(y,x)\right)\, {\rm d}t_1\, {\rm d}t_2\, .
\end{align*}
Let us denote $\Sigma(x,y)$ for this limit. 
It follows from Assumption~(v) that, for almost every $x$,
\[
\lim_{y\rightarrow x}p_{t_1,t_2}(x,x)+p_{t_1,t_2}(y,y)-p_{t_1,t_2}(x,y)-p_{t_1,t_2}(y,x)=0\, .
\]
Furthermore, Assumptions~(iii) and~(iv) 
ensure that
\[
\forall u_1,u_2>0,\quad \forall x,y\in\mathbb R,\quad 
p_{u_1,u_1+u_2}(x,y)\le \frac{2C}{(u_1u_2)^\alpha}\, .
\]
Hence, it follows from the dominated convergence theorem that
\[
\lim_{y\rightarrow x}\Sigma(x,y)=0\, .
\]
This ends the proof of Condition~{\bf (RW2.b)} of~\cite{DGP-EJP-2009} and so of Assumption~(b) of Proposition~\ref{plocaltime},  which ensures the
 convergence in distribution of~\eqref{FDD} to~\eqref{FDD0}.\\
It remains to prove the tightness of
\[
\left(\left(n^{-1+\alpha(1-\frac 1\beta) }\sum_{k=1}^{\lfloor nt\rfloor}\xi_{X_{k-1}}\right)_{t\geq 0}\right)_{n\in\mathbb N^*}
\, \]
in $\mathcal D$.\\
We proceed as in~\cite{KS-1979} with the adaptation provided in~\cite{FrankePeneWendler2} related to the convergence in $\mathcal D(0,T)$.
Recall that the normalization $\delta$ exponent is given by
\[
\delta=1-\alpha\left(1-\frac 1\beta\right)\, .
\]
We start by the easiest case where $\beta=2$. Then $\delta=1-\frac\alpha 2$. 
Due to the classical tightness criteria given in \cite[Theorem 13.5]{Billingsley}, since the limit process is continuous, it suffices to prove
that there exists $K>0$ such that for all $t,t_1,t_2\in[0,T],
T<\infty,$ s.t. $t_{1}\leq t\leq t_{2},$ for all $n\geq 1$,
\begin{equation}\label{pro}
\bb E\Big[|Y_{\lfloor nt\rfloor}-Y_{\lfloor nt_1\rfloor}| \cdot \ |Y_{\lfloor nt_2\rfloor}-Y_{\lfloor nt\rfloor}|\Big]\leq K
 (n|t_{2}-t_{1}|)^{2-\alpha},
\end{equation}
(see~\cite[(13.14)]{Billingsley} noticing that $2-\alpha>1$). 
Using Cauchy-Schwarz inequality, it is enough to prove that there exists $K'>0$ such that for all $n\geq 1$, for all $t_1\leq t$,
\begin{equation}\label{pro1-red}
\bb E\left[\left(Y_{\lfloor nt\rfloor}-Y_{\lfloor nt_1\rfloor}\right)^2\right]\leq K'
(\lfloor nt\rfloor-\lfloor nt_1\rfloor)^{2-\alpha}.
\end{equation}
Indeed if $t_2-t_1<n^{-1}$, then 
$\lfrf{nt_2}\in \{\lfrf{nt_1},\lfrf{nt_1}+1\}$
hence 
$\lfloor nt\rfloor=\lfloor nt_1\rfloor$
 or $\lfloor nt\rfloor=\lfloor nt_2\rfloor$ and so
\[
\bb E\Big[|Y_{\lfloor nt\rfloor}-Y_{\lfloor nt_1\rfloor}| \cdot \ |Y_{\lfloor nt_2\rfloor}-Y_{\lfloor nt\rfloor}|\Big]=0\, ,
\]
whereas if $t_2-t_1\ge n^{-1}$, 
then~\eqref{pro1-red} will imply that
\[\bb E\left[\left(Y_{\lfloor nt_j\rfloor}-Y_{\lfloor nt\rfloor}\right)^2\right]\le K'
(\lfloor nt_2\rfloor-\lfloor nt_1\rfloor)^{2-\alpha}\\
\le 
 K'
\left( 2 n(t_2- t_1)\right)^{2-\alpha}.\]
To prove \eqref{pro1-red}, we set $m=\lfrf{nt}$
and $l=\lfrf{nt_1}$. We observe that
since the $\xi$'s are independent and centered, we have
\begin{align}\nonumber
\bb E\left[\left(Y_{m}-Y_{l}\right)^2\right]
&=\mathbb E[\xi_0^2]\sum_{k,j=l}^{m-1} \bb P\left(X_{k}=X_{j} \right) \\
\nonumber&= \mathbb E[\xi_0^2]
\pare{m-l+2\sum_{k=l}^{m-2} \sum_{j=k+1}^{m-1}  \bb P\left(X_{j}-X_{k}=0 \right)}\\
&=\mathbb E[\xi_0^2] \pare{m-l+2\sum_{k=l}^{m-2} \sum_{j=1}^{m-1-k}  \bb P\left(X_{j} =0\right)}\, ,\label{BBB0}
\end{align}
since the increments of the PAPA are stationary. It follows from Assumption~(ii) that
\begin{align}
\nonumber\sum_{k=l}^{m-2} \sum_{j=1}^{m-1-k}  \bb P\left(X_{j} =0\right)
&=\mathcal O\left(\sum_{k=l}^{m-2} \sum_{j=1}^{m-1-k} j^{-\alpha}\right)\\
\nonumber&=\mathcal O\left(\sum_{k=l}^{m-2}  (m-k-1)^{1-\alpha}\right)\\
\nonumber&=\mathcal O\left(\sum_{k=1}^{m-l}  k^{1-\alpha}\right)\\
&=\mathcal O\left((m-l)^{2-\alpha}\right)\, ,\label{BBB1}
\end{align}
which ends the proof of \eqref{pro1-red}.

This ends the proof of~\eqref{pro} and so of the tightness when $\beta=2$.
\\
Assume from now on  that $\beta\in(1,2)$. 
We adapt the argument of tightness~\cite{KS-1979} which is based on the criteria given in~\cite{Billingsley}.
As in~\cite[(12.27) p. 131]{Billingsley}, for any function $G\in D([0,T],\mathbb R)$, we define
\[
w''_n\left(G,h\right):=\sup_{0\le t_1<t<t_2\le T,|t_2-t_1|<h}
\left(\min \left(n^{-\delta}\left|G(t_2)-G(t)\right|,n^{-\delta}\left|G(t)-G(t_1)\right|\right)\right)\, .
\]
Since the limit process is continuous, due to~\cite[Theorem 13.3]{Billingsley}, the tightness will follow from the fact 
\[
\forall\varepsilon>0,\ \forall \eta>0,\quad \exists h>0,\ \limsup_{n\rightarrow +\infty} 
\mathbb P\left(w''_n\left(Y_{\lfloor n\cdot\rfloor},h\right)>\eta\right)<\varepsilon \, .\]

We fix $\eta>0$ and $\varepsilon>0$. 
We will decompose $n^{-\delta}Y_{\lfloor n\cdot\rfloor}$ into a sum of three processes as follows~
\[
Y_{\lfloor nt\rfloor}=\overline Y_{\lfloor nt\rfloor}^{(\varepsilon)}+tE_n^{(\varepsilon)}
+F_{n}^{(\varepsilon)}(t)\, ,
\]
where $(F_n^{(\varepsilon)}(\cdot))_{n\geq 1}$ satisfies
\begin{equation}\label{ESTI3}
\limsup_{n\rightarrow + \infty}\mathbb P\left(\sup_{t\in[0;T]}n^{-\delta}\left|F_{n}^{(\varepsilon)}(t)\right|>\eta/4\right)<\varepsilon/2\, ,
\end{equation}
where $\left(n^{-\delta}E_n^{(\varepsilon)}\right)_{n\ge 1}$ is a sequence of uniformly bounded real numbers, and so there exists $h>0$ small enough such that
\begin{equation}\label{ESTI2}
n^{-\delta} \left|E_n^{(\varepsilon)}\right| h\le \eta/4\, ,
\end{equation}
and finally, we will prove that,
for $h$ small enough,
\begin{equation}\label{ESTI1}
\limsup_{n\rightarrow +\infty}\mathbb P\left(w_n''\left(\overline Y_{\lfloor n\cdot\rfloor}^{(\varepsilon)},h\right)>\eta/4\right)<\varepsilon/2\, .
\end{equation}
It will follow that for $h$ small enough
so that~\eqref{ESTI2} holds true, we have
\begin{equation}\label{EST}
\limsup_{n\rightarrow +\infty}\mathbb P\left(w_n''\left( Y_{\lfloor n\cdot\rfloor},h\right)>\eta\right)\le \eqref{ESTI1}+\eqref{ESTI3}<\varepsilon\, .
\end{equation}
Indeed
\begin{align}\label{EQ0}
w''_n\left(Y_{\lfloor n\cdot\rfloor},h\right)
\le w''_n\left(\overline Y^{(\varepsilon)}_{\lfloor n\cdot\rfloor},h\right)+2
\sup_{t\in[0;T]}n^{-\delta}\left|F_{n}^{(\varepsilon)}(t)\right|
+n^{-\delta} \left|E_n^{(\varepsilon)}\right| h\, .
\end{align}

This follows from the fact that
 \[
\left|Y_{\lfloor nt_j\rfloor}-Y_{\lfloor nt\rfloor}\right|
\le \left|\overline Y^{(\varepsilon)}_{\lfloor nt_j\rfloor}-\overline Y^{(\varepsilon)}_{\lfloor nt\rfloor}\right|+\left|E_n^{(\varepsilon)}(t_j-t)\right| +
\left|F_{n}^{(\varepsilon)}(t)\right|+\left|F_{n}^{(\varepsilon)}(t_j)\right|\, .
\]
Thus, due to~\eqref{EQ0}, the condition $w''_n\left(Y_{\lfloor n\cdot\rfloor},h\right)>\eta$ implies that one of the three following quantities is larger than  $\eta/4$~:
\begin{align*}
    &w_n''\left( \overline Y^{(\varepsilon)}_{\lfloor n\cdot\rfloor},h\right)\\
    &n^{-\delta} \left|E_n^{(\varepsilon)}h\right|  \\
    &\sup_{t\in[0;T]}n^{-\delta}\left|F_{n}^{(\varepsilon)}(t)\right|\, ,
\end{align*}
this implies~\eqref{EST} provided~\eqref{ESTI3} holds true.
\begin{itemize}
\item We start by proving that,
\begin{equation}\label{AAAX}
\forall\varepsilon>0,\ \exists B_\varepsilon>1,\quad 
\forall n\in\mathbb N^*,n>1/T\quad \mathbb P\left(\sup_{k=0,...,\lfloor nT\rfloor} |X_k|> B_\varepsilon \lfloor nT\rfloor^{\alpha}\right)\le\varepsilon\, .
\end{equation}
To this end we use the convergence in distribution of \[\left(\lfloor nT\rfloor^{-\alpha}\sup_{k=0,...,\lfloor nT\rfloor} |X_k|=\lfloor nT\rfloor^{-\alpha}\sup_{s\in[0,T]} |X_{\lfloor ns\rfloor}|\right)_{n\in\mathbb N^*}\] 
to $\sup_{s\in [0,T]}|\Delta_s|$, which implies the tightness of $\left(n^{-\alpha}\sup_{k=0,...,\lfloor nT\rfloor} |X_k|\right)_{n\in\mathbb N^*}$ and so~\eqref{AAAX}.
\item For any $ \varepsilon>0 $, we consider $B_{\varepsilon/4}>0$ satisfying~\eqref{AAAX}
where we replace $\varepsilon$ by $\varepsilon/4$
and we consider $ \rho_\varepsilon>0 $ such that for all $ n\in\mathbb{N}^* $ one has
\begin{eqnarray} \label{rho}
3B_{\varepsilon/4}\lfloor nT\rfloor^{\alpha}\mathbb P\left(|\xi_0|>\rho_\varepsilon  n^{\frac{\alpha}{\beta}}\right)
    <\frac{\varepsilon}{4}.
\end{eqnarray}
The existence of $\rho_\varepsilon$ comes from the fact that there exists some constant $C>0$ such that
\begin{equation}\label{asymp}
    \mathbb P(|\xi_0|>u)\sim C u^{-\beta}\quad \mbox{ as }u\rightarrow +\infty\, .
\end{equation}
This estimate  follows from the fact that the distribution of $\xi_0$ is in the 
normal domain of attraction of a $\beta$-stable distribution (see~\cite{GnedenkoKolmogorov}).
This implies the existence of a constant $K>0$ such that
\[
\forall u>0,\quad 
\mathbb P(|\xi_0|>u)\le K u^{-\beta}\, ,
\]
and so, for all $\rho>0$,
\begin{align*}
3B_{\varepsilon/4}\lfloor nT\rfloor^{\alpha}\mathbb P\left(|\xi_0|>\rho n^{\frac{\alpha}{\beta}}\right)
&\le    3B_{\varepsilon/4}\lfloor nT\rfloor^{\alpha} K\rho^{-\beta}n^{-\alpha}\\
&\le 3B_{\varepsilon/4} T^{\alpha} K\rho^{-\beta}\, ,
\end{align*}
and the above quantity goes to 0, uniformly in $n\in\mathbb N^*$, as $\rho\rightarrow 0+$.
\item Truncation. We define 
\[
  \bar{\xi}_x^{(\varepsilon)}:=\xi_x\mathbf{1}_{\left\{| \xi_x|\leq \rho_\varepsilon n^{\frac{\alpha}{\beta}}\right\}}\, .
\]
Since $\xi$ is centered, 
\begin{align*}
\left| \mathbb{E}\left[\bar \xi_x^{(\varepsilon)}\right]\right|&
= \left| \mathbb{E}\left[\bar\xi_x^{(\varepsilon)}-\xi_x\right]\right|= \left| \mathbb{E}\left[\xi_x\mathbf 1_{\{|\xi_x|>\rho_\varepsilon  n^{\frac\alpha\beta}\}}\right]\right| \\
&\le  \mathbb{E}\left[|\xi_x|\mathbf 1_{\{|\xi_x|>\rho_\varepsilon  n^{\frac\alpha\beta}\}}\right] \\
&\le \int_0^{+\infty}\mathbb P\left( |\xi_x|\mathbf 1_{\{|\xi_x|>\rho_\varepsilon  n^{\frac\alpha\beta}\}}>u\right)\, {\rm d}u\, .
\end{align*}
Hence, using again~\eqref{asymp}, we obtain that
\begin{align}\nonumber
\left| \mathbb{E}\left[\bar \xi_x^{(\varepsilon)}\right]\right|
&\le \int_0^{\rho_\varepsilon n^{\frac\alpha\beta}}\mathbb P\left( |\xi_x|>\rho_\varepsilon n^{\frac\alpha\beta}\right)\, {\rm d}u+
\int_{\rho_\varepsilon n^{\frac\alpha\beta}}^{+\infty}\mathbb P\left( |\xi_x|>u\right)\, {\rm d}u\\
&=\mathcal O\left( n^{(1-\beta)\frac{\alpha}{\beta}}\right)\, . \label{BBB3}
\end{align}
We set
\[  \overline Y_n^{(\varepsilon)}:= \sum_{k=0}^{n-1}\left(\overline \xi_{X_k}^{(\varepsilon)}-\mathbb E\left[\overline \xi^{(\varepsilon)}_0\right]\right)\,  .\]
Let us prove estimate~\eqref{ESTI1}.
Due to~\cite[Formula (13.14)]{Billingsley} the proof of \cite[Theorem 13.5]{Billingsley} (see the last centered formula of this proof in \cite[p.143]{Billingsley}), 
this estimate will follow from  
the existence of $K'_\varepsilon>0$ such that, for every $0<t_1<t<t_2<T$ and every $n\in\mathbb N^*$,
\begin{equation}\label{AAA000}
   \mathbb{E}\left[\left|\overline{Y}^{(\varepsilon)}_{\lfloor nt\rfloor}-\overline{Y}^{(\varepsilon)}_{\lfloor nt_1\rfloor}\right|\, \left|\overline{Y}^{(\varepsilon)}_{\lfloor nt\rfloor}-\overline{Y}^{(\varepsilon)}_{\lfloor nt_2\rfloor}\right|\right] 
   \leq  K'_\varepsilon
n^{2\delta}\left(t_2- t_1\right)^{2-\alpha}\, .
\end{equation}
Again, as in the case $\beta=2$, we observe that the left hand side of~\eqref{AAA000} is null if $t_2-t_1<1/n$ (because $\lfloor nt\rfloor$ is equal to $\lfloor nt_1\rfloor$ or to $\lfloor nt_2\rfloor$).\\
It remains to treat the case where $t_2-t_1\geq 1/n$.
Recall that in this case 
\[\left|\lfloor nt\rfloor-\lfloor nt_j\rfloor\right|\le \lfloor nt_2\rfloor-\lfloor nt_1\rfloor\le 2n(t_2-t_1)\, .
\] 
Therefore,~\eqref{AAA000} will follow via the Cauchy-Schwarz inequality from the fact that there exists $K''_\varepsilon>0$ such that, for any $0<t_1<t<t_2<T$ and $n\in\mathbb N^*$, we have
\begin{align}\label{AAA00}
   \mathbb{E}\left[\left(\overline{Y}^{(\varepsilon)}_{\lfloor nt\rfloor}-\overline{Y}^{(\varepsilon)}_{\lfloor nt_j\rfloor}\right)^2\right] 
   &\leq  K''_\varepsilon
n^{2\frac\alpha\beta-\alpha}\left|\lfloor nt\rfloor-\lfloor nt_j\rfloor\right|^{2-\alpha}\\
&\leq  K''_\varepsilon
n^{2\delta}\left(2(t_2- t_1)\right)^{2-\alpha}
\, ,\label{AAA00b}
\end{align}
where we used the fact that
\[
2\delta=2-2\alpha\left(1-\frac 1\beta\right)
=2\frac\alpha\beta-\alpha+2-\alpha \, .
\]
Applying again~\eqref{BBB0},~\eqref{BBB1} with $m=\lfloor nt\rfloor$ and with $l=\lfloor nt_j\rfloor$, 
we obtain the existence of $K'''>0$ such that, for every $0<t_1<t<t_2<T$ and $n\in\mathbb N^*$, we have
\begin{equation}\label{KKK}
\mathbb{E}\left[\left(\overline{Y}^{(\varepsilon)}_{\lfloor nt\rfloor}-\overline{Y}^{(\varepsilon)}_{\lfloor nt_j\rfloor}\right)^2\right]\le 
K''' \mathbb E\left[\left(\overline\xi^{(\varepsilon)}_0-\mathbb E[\overline\xi^{(\varepsilon)}_0]\right)^2\right]\left|\lfloor nt\rfloor-\lfloor nt_j\rfloor\right|^{2-\alpha}\, .
\end{equation}
Furthermore
\begin{align}\nonumber
    \mathbb E\left[\left(\overline\xi^{(\varepsilon)}_0-\mathbb E[\overline\xi^{(\varepsilon)}_0]\right)^2\right]&\le\mathbb E\left[(\overline\xi^{(\varepsilon)}_0)^2\right]=
    \mathbb E\left[\xi_0^2\indic_{\{| \xi_0|\leq \rho_\varepsilon n^{\frac{\alpha}{\beta}}\}}\right]\\
    \nonumber&=\int_0^{+\infty}\mathbb P\left(\xi_0^2 \indic_{\{| \xi_0|\leq \rho_\varepsilon n^{\frac{\alpha}{\beta}}\}}>u\right)\, {\rm d}u\\
    \nonumber&\leq \int_0^{\rho_\varepsilon^2 n^{\frac{2\alpha}{\beta}}}\mathbb P\left(\xi_0^2>u\right)\, {\rm d}u\\
    &=\mathcal O\left( n^{2\frac\alpha\beta -\alpha} \right),\, \label{KKK1}
\end{align}
using (\ref{asymp}).
Combining the two estimates~\eqref{KKK} and~\eqref{KKK1}, we obtain \eqref{AAA00} and~\eqref{AAA00b}. 
Therefore we have proved
the third key estimate~\eqref{ESTI1}.
\item Linear part. 
It follows from~\eqref{BBB3} that 
$$ E_n^{(\varepsilon)}:=\mathbb{E}\left[\sum_{k=0}^{n-1}\overline{\xi}_{X_k}^{(\varepsilon)}\right]=n\mathbb E[\overline \xi^{(\varepsilon)}_0]=\mathcal O\left(n^{1+(1-\beta)\frac\alpha\beta}\right)=\mathcal O(n^{\delta})\, ,$$
since $\delta=1-\alpha\left(1-\frac 1\beta\right)=1+(1-\beta)\frac{\alpha}\beta$. 
This implies the second key estimate~\eqref{ESTI2}.
\item Error term. 
we have
\[
Y_{\lfloor nt\rfloor}-\overline Y^{(\varepsilon)}_{\lfloor nt\rfloor}-tE_n^{(\varepsilon)} =\left(nt-\lfloor nt\rfloor\right)\mathbb E[\overline \xi^{(\varepsilon)}_0]+ \sum_{k=0}^{\lfloor nt\rfloor-1}\left(\xi_{X_k}-\overline \xi^{(\varepsilon)}_{X_k}\right)\, .
\]
It follows from~\eqref{BBB3} that
 $ \mathbb{E}[\overline\xi^{(\varepsilon)}_0]=\mathcal O\left(n^{(1-\beta)\frac{\alpha}{\beta}}\right)=\mathcal O\left(n^{\delta-1}\right) $ and so that
$$
\sup_{t\in[0,T]}
n^{-\delta}
\left(nt-\lfloor nt\rfloor\right) |\mathbb E[\overline \xi^{(\varepsilon)}_0] |     \le n^{-\delta} |\mathbb E[\overline \xi^{(\varepsilon)}_0]|=\mathcal O\left(n^{-1}\right).$$
For $n$ large enough, $$
\sup_{t\in[0,T]}
n^{-\delta}
\left(nt-\lfloor nt\rfloor\right) |\mathbb E[\overline \xi^{(\varepsilon)}_0] |     \le \eta/4$$
and so
\begin{align*}
 & \limsup_{n\rightarrow\infty}\mathbb{P}\left(\sup_{0\leq t\leq T}n^{-\delta}\left|Y_{\lfloor nt\rfloor}-\overline Y^{(\varepsilon)}_{\lfloor nt\rfloor}-t E_n^{(\varepsilon)}  \right|>\frac{\eta}{2}\right) \\
   & \leq   \limsup_{n\rightarrow\infty}\mathbb{P}\left(
     \sup_{0\leq t\leq T}
     n^{-\delta}\left|\sum_{k=0}^{\lfloor nt\rfloor-1}(\xi_{X_k}-\overline \xi^{(\varepsilon)}_{X_k}) \right| >\frac{\eta}{4}\right)\\
&\leq 
 \limsup_{n\rightarrow\infty}\mathbb{P}\left(\exists k=0,...,\lfloor nT\rfloor-1\, :\, \xi_{X_k}\ne \overline\xi^{(\varepsilon)}_{X_k} \right)  \\
 &\leq 
 \limsup_{n\rightarrow\infty}\mathbb{P}\left(\exists k=0,...,\lfloor nT\rfloor-1\, :\, |\xi_{X_k}|>\rho_\varepsilon n^{\frac\alpha\beta} \right) \, .
 \end{align*}
Therefore, this quantity is dominated by
  \begin{align*}
   &\leq  
  \limsup_{n\rightarrow\infty} \mathbb{P}\left(\exists x\in\mathbb{Z}: |x|\leq B_{\varepsilon/4} \lfloor nT\rfloor^{\alpha},\, 
  |\xi_{x}|>\rho_\varepsilon n^{\frac\alpha\beta}\right) \\
 &+  \ \limsup_{n\rightarrow\infty} \mathbb{P}\left(
      \sup_{k=0,...,\lfloor nT\rfloor-1} |X_k|> B_{\varepsilon/4} \lfloor nT\rfloor^{\alpha}\right) \\
  &\leq   \limsup_{n\rightarrow\infty} \ 3B_{\varepsilon/4}\lfloor nT\rfloor^{\alpha}\mathbb{P}\left(|\xi_{0}|>\rho_\varepsilon n^{\frac\alpha\beta}\right) 
          + \frac{\varepsilon}{4}\\
&\leq \frac{\varepsilon}{2}\, ,
\end{align*}
for every integer $n>1/T$,
where we used the definition of $B_{\varepsilon/4}$ in~\eqref{AAAX} and the property~\eqref{rho}.
This ends the proof of the first key estimate~\eqref{ESTI3} and so of the tightness in the case $\beta\in(1,2)$.
\end{itemize}
\end{proof}

\section{Joint Convergence of the PAPAPA together with the PA and the PAPA}\label{PAPAPA}
We consider here the PAPAPA in a more general context than the one presented at the beginning of the manuscript.
Let $(\eta_k)_{k\in\mathbb N^*}$, $(\xi^{(1)}_z)_{z\in\mathbb Z}$ and  $(\xi^{(2)}_x)_{x\in\mathbb Z}$ be three independent sequences of independent identically distributed random variables
satisfying the following assumptions
\begin{itemize}
\item $\eta_1$ is $\mathbb Z$-valued, centered and admit moments of any order,
\item $\xi^{(1)}_0$ is $\mathbb Z$-valued, centered and square integrable,
\item the distribution of $\xi^{(2)}_0$ belongs to the normal domain of attraction of the stable distribution $\mathcal{S}_{\beta}(\sigma,\nu,0)$ with characteristic function given in (\ref{eq1.05}).
\end{itemize}
Let $(Z_n)_{n\in\mathbb N}$ be the random walk with 
steps $(\eta_k)_{k\in\mathbb N^*}$, i.e.
\[
Z_n:=\sum_{k=1}^n\eta_k\, ,
\]
with the usual convention $Z_0:=0$.
Let $(X_n)_{n\in\mathbb N}$ be the PAPA given by
\[
X_n:=\sum_{k=1}^n\xi^{(1)}_{Z_{k-1}}\, .
\]
Let $d_0$ be the greatest common divisor of the support of the distribution of 
$\eta_1$. As noticed in Remark~\ref{notcoprime}, the process $X_n$
can be rewritten
\[
X_n=\sum_{k=1}^n\xi^{(1)}_{d_0Z_{k-1}/d_0}\, .
\]
So $X$ can be seen as the PAPA with steps  $\left(\eta_k/d_0\right)_{k\in\mathbb N^*}$  and with scenery $\left(\xi^{(1)}_{d_0z}\right)_{z\in\mathbb Z}$.
Therefore it follows from Proposition~\ref{PRO1} 
applied with $M=1$ and with $X'=Z_{\lfloor n\cdot\rfloor}/(\sqrt{n}d_0)$,
$\xi=\xi^{(1)}_{d_0\cdot}$  
that
\begin{equation}\label{KS}
\left((Z_{\lfloor nt\rfloor}/(\sqrt{n}d_0))_{t\ge 0},n^{-\frac 34 }X_{\lfloor nt\rfloor}\right)_{t\ge 0}
\xrightarrow[n\to\infty]{\cal L}\left(B, \int_{\mathbb R}L^B(t,x)\, {\rm d}W(x)\right)_{t\ge 0}\, ,
\end{equation} where $L^B$ is the local time of a real Brownian motion $B$ and $W$ a bilateral Brownian motion, independent of $B$. The processes $B$ and $W$ have respective variances $\mathbb E[\eta_1^2]/d_0^2$ and $\mathbb E[(\xi_0^{(1)})^2]$. 

The existence and continuity of the local time of the process $\Delta$ was proved in~\cite{CGPPS-AOP-2014} and will be denoted $L^{\Delta}.$
Let $(Y_n)_{n\in\mathbb N}$ be the PAPAPA given by
\[
Y_n:=\sum_{k=1}^n\xi^{(2)}_{X_{k-1}}\, .
\]
Let $d_1$ be the greatest common divisor of the support of the distribution of $\xi_0^{(1)}$.
Again, $Y_n$ can be rewritten as
\[
Y_n:=\sum_{k=1}^n\xi^{(2)}_{d_1X_{k-1}/d_1}\, ,
\]
as the PAPAPA with increments $\left(\eta_k/d_0\right)_{k\in\mathbb N^*}$ and with successive sceneries $\left(\xi^{(1)}_{d_0z}/d_1\right)_{z\in\mathbb Z}$ and $\left(\xi^{(2)}_{d_1x}\right)_{x\in\mathbb Z}$. 
\begin{thm}\label{THMtriple}
The family of processes 
\[
\left(\left(n^{-\frac 34}  X_{\lfloor nt\rfloor}\right)_t,\left(n^{-\frac{(\beta+3)}{4\beta}}  Y_{\lfloor nt\rfloor}\right)_t, \left(n^{-\frac 12}  Z_{\lfloor nt\rfloor}\right)_t\right)_{n\in\mathbb N^*}
\]
converges in distribution to  the joint process 
$\left(d_1\Delta,\Gamma, d_0 B\right)$
given by
\[\Gamma(t)=\int_{\bb R} L^\Delta(t,x)\,dU(x)\, ,\]
recalling that
\[
\Delta(t)=\int_{\bb R} L^B(t,x)\,dW^{(1)}(x)\, ,
\]
where $B$, $W^{(1)}$ are two independent Brownian motions with respective variance $Var(\eta_1/d_0)$, $Var(\xi_0^{(1)}/d_1)$, $L^B$ and $L^\Delta$ are the respective local times of $B$ and $\Delta$, and $U$ is a bilateral $\beta$-stable Lévy process independent of $B$ and $W^{(1)}$.
\end{thm}
Observe that $(B,d_1\Delta)$ corresponds to the limit appearing on the right hand side of~\eqref{KS}
up to taking $W^{(1)}=W/d_1$, indeed
\[
\int_{\mathbb R}L^B(t,x)\, {\rm d}W(x)=d_1\int_{\mathbb R}L^B(t,x)\, {\rm d}W^{(1)}(x)=d_1\Delta\, .
\] 
\begin{proof}[Proof of Theorem~\ref{THMtriple}]
We apply Proposition~\ref{PRO1} with $A_n=Z_{\lfloor n\cdot\rfloor}/(\sqrt{n}d_0)$, with $\alpha=\frac 34$, and with the stationary increments process  $X$ replaced by $X/d_1$ and with the scenery $(\xi_x=\xi^{(2)}_{d_1 x})_{x\in\mathbb Z}$.
This will imply the convergence in distribution of
\[
\left(\left(n^{-\frac 34}  X_{\lfloor nt\rfloor}/d_1\right)_t,\left(n^{-\frac{(\beta+3)}{4\beta}}  Y_{\lfloor nt\rfloor}\right)_t, \left(n^{-\frac 12}  Z_{\lfloor nt\rfloor}/d_0\right)_t\right)_{n\in\mathbb N^*}
\]
to  the joint process 
$\left(\Delta,\Gamma,  B\right)$, and so the announced convergence.
Assumption~(i) comes from~\eqref{KS}.  Assumption~(ii) was proved in~\cite[Theorem 1]{CGPPS-AOP-2011}. Assumption~(iv)
was proved in the second part of~\cite[Proposition 27]{Pene-EJP2021},
Assumption~(v) was proved in~\cite[Corollary 2]{CGPPS-AOP-2014}. Finally Assumption~(iii) will follow from Proposition~\ref{PRO2} below, which is an adaptation of the proof of the first part of ~\cite[Proposition 27]{Pene-EJP2021}.
\end{proof}
Set $X''_n:=X_n/d_1$. Observe that $X''_1$ has the same distribution as $\xi_0^{(1)}/d_1$. 
Let $d$ be the greatest common divisor of the set $\{a-b\, :\,  \mathbb P( X''_1 = a)  \mathbb P( X''_1 = b) > 0\}$ and let $a\in \mathbb Z$ such that $\mathbb P(X''_1=a)>0$. Then 
 \[\mathbb P\left(X''_1\in a+d\mathbb Z\right)=1\, .\]
 Since $d_1$ is the greatest common divisor of the support of the distribution of $\xi^{(1)}_0$, 
it follows that the support of the distribution of $X''_1$ generates the group $\mathbb Z$. This ensures that $a$ and $d$ are coprime.\\
 In other words, one can characterize $a$ and $d$ as follows : writing $\varphi_1$ for the characteristic function of $X''_1$, $d$ is also the unique positive integer such that  
 \[\{u\in\mathbb R \, :\; |\varphi_1(u)|=1\} =\frac{2\pi}d\mathbb Z\]
 and $a$ satisfies $\varphi_1(2\pi/d)=e^{\frac {2i\pi a}d}$, i.e. $e^{\frac{2i\pi X''_1}d}=e^{\frac {2i\pi a}d}$ almost surely, and $e^{\frac {2i\pi a}d}$
 is a a primitive $d$-th root of the unity.
 
\begin{prop}\label{PRO2}
For any $x_1,x_2\in\bb R$ and $t_1,t_2>0$, 
\begin{multline}\label{cvxyn2}
n^{\frac 32}\bb P\left(X''_{n_1}=\lfloor n^{\frac 34}x_1\rfloor,X''_{n_1+n_2}=\lfloor n^{\frac 34}x_2\rfloor\right)\\
\sim  \indic_{d\mathbb Z^2}(n_1a-\lfloor n^{3/4} x_1\rfloor,(n_1+n_2)a-\lfloor n^{3/4} x_2\rfloor)
d^2p_{t_1,t_1+t_2}(x_1,x_2)\, ,
\end{multline}
as $n\to\infty$, $\frac{n_i}{n}\to t_i$
where $p_{T_1,T_2}$ is the density of the couple of random variable  $(\Delta_{T_1},\Delta_{T_2})$.
\end{prop}
\begin{proof}
Recall that $(X''_n)_n$ 
is the PAPA with random walk $(S_n:=Z_n/d_0)_n$ on $\mathbb Z$
and with random scenery $\left(\xi'_z:=\xi_{d_0 z}^{(1)}/d_1\right)$. 
We adapt the proof of~\cite[Theorem 5]{CGPPS-AOP-2014} in which the case where the $x_j$'s are null was investigated.
We start by setting some notations. 
Let $t_1,t_2$ be two positive real numbers and $n_1,n_2$ be two positive integers. Let us write $N_m(z)$ for the local time at time $m$ and at position $z$ of the random walk $S$. We set 
$N_{n_1}^{(1)}(x)=N_{n_1}(x)$ and  $N_{n_2}^{(2)}(x)=N_{n_1+n_2}(x)-N_{n_1}(x)$\, .
We consider the matrix
\[A_{n_1,n_2}=\left(\left\langle N_{n_i}^{(i)},N_{n_j}^{(j)}\right\rangle_{\ell^2(\mathbb Z)}\right)_{i,j=1,2}\, ,\]
i.e.
\[A_{n_1,n_2}=
\begin{pmatrix}
\left\langle N_{n_1},N_{n_1}\right\rangle_{\ell^2(\mathbb Z)} & 
\left\langle N_{n_1},N_{n_1+n_2}-N_{n_1}\right\rangle_{\ell^2(\mathbb Z)} \\
\left\langle N_{n_1},N_{n_1+n_2}-N_{n_1}\right\rangle_{\ell^2(\mathbb Z)} &
\left\langle N_{n_1+n_2}-N_{n_1},N_{n_1+n_2}-N_{n_1}\right\rangle_{\ell^2(\mathbb Z)}
\end{pmatrix}
\, .
\]
We also define
\[D_{n_1,n_2}=\det A_{n_1,n_2}\, .\]
Writing $L_t=L^B(t,\cdot)$ and setting $T_1:=t_1$ and $T_2:=t_1+t_2$,
we consider also their continuous analogues defined as follows
\[ M_{T_1,T_2}=\left(\left\langle L_{T_i},L_{T_j}\right\rangle_{{\mathbf L}^2(\mathbb R)}\right)_{i,j=1,2},
\qquad
\cal D_{T_1,T_2}=\det M_{T_1,T_2},\]
\[\tilde M_{t_1,t_2}=\begin{pmatrix}
\left\langle L_{T_1},L_{T_1}\right\rangle_{{\mathbf L}^2(\mathbb R)} & \left\langle L_{T_1},L_{T_2}-L_{T_1}\right\rangle_{{\mathbf L}^2(\mathbb R)}\\
\left\langle L_{T_1},L_{T_2}-L_{T_1}\right\rangle_{{\mathbf L}^2(\mathbb R)} & \left\langle L_{T_2}-L_{T_1},L_{T_2}-L_{T_1}\right\rangle_{{\mathbf L}^2(\mathbb R)}
\end{pmatrix},
\qquad
\tilde{\cal D}_{t_1,t_2}=\det \tilde M_{t_1,t_2}\, .\]
We observe that $\tilde{\cal D}_{t_1,t_2}=\cal D_{T_1,T_2}$.
Setting
$$\varphi_{n_1,n_2}(\theta):=\bb E\left[e^{i(\theta_1 X''_{n_1}+\theta_2(X''_{n_1+n_2}-X''_{n_1}))}\right]\, ,$$
and
$a_n(x)=([n^{3/4}x_1],[n^{3/4}x_2]-[n^{3/4}x_1])$,
we have
\[
\bb P\left(X''_{n_1}=[n^{3/4}x_1],X''_{n_1+n_2}=[n^{3/4}x_2]\right)
=\frac{1}{(2\pi)^2}\int_{[-\pi ,\pi]^2}
\varphi_{n_1,n_2}(\theta)\exp(-i\left\langle \theta,a_n(x)\right\rangle)\,d\theta\, .
\]
Let us explain how to adapt~\cite[Lemma 15]{CGPPS-AOP-2014}. 
Since $X''$ is a PAPA with random walk $S$ and random scenery $(\xi'_z)_{z\in\mathbb Z}$,  we observe that
\begin{align*}
\varphi_{n_1,n_2}(\theta)&=\bb E\left[e^{i\sum_{z\in\mathbb Z}\xi'_z(\theta_1 N_{n_1}(z)+\theta_2(N_{n_1+n_2}(z)-N_{n_1}(z))}\right]
.\end{align*}
Since, conditionally to $S$, $(\xi'_z)_{z\in\mathbb Z}$
is a sequence of independent identically distributed random variables of characteristic function $\varphi_{1}$, it follows that
\begin{align*}
\varphi_{n_1,n_2}(\theta)&=\bb E\left[\prod_{z\in\mathbb Z}\varphi_{1}\left(\theta_1 N_{n_1}(z)+\theta_2(N_{n_1+n_2}(z)-N_{n_1}(z))\right)\right]
\, .\end{align*}
Since $e^{2i\pi\xi'_0/d}=\lambda:=\varphi_{1}(2\pi/d)=e^{2i\pi a/d}$ almost surely, the following relation holds true
\[
\varphi_{1}\left(\frac{2k\pi}d+u\right)=\lambda^k
\varphi_1(u)\, .
\]
Thus
\begin{align*}
\varphi_{n_1,n_2}&\left(\theta+\frac{2\pi}d(k_1,k_2)\right)\\
&=\mathbb E\left[\prod_{z\in\mathbb Z}\varphi_{1}\left(\theta_1 N_{n_1}^{(1)}(z)+\theta_2N_{n_2}^{(2)}(z)\right)\lambda^{k_1N_{n_1}^{(1)}(z)+k_2N_{n_2}^{(2)}(z)}
\right]\\
&=\lambda^{k_1n_1+k_2n_2}\mathbb E\left[\prod_{z\in\mathbb Z}\varphi_{1}\left(\theta_1 N_{n_1}^{(1)}(z)+\theta_2N_{n_2}^{(2)}(z)\right)
\right]\, ,
\end{align*}
where we used the fact that $\sum_{z\in\mathbb Z}N_{n_1}(z)=n_1$ and $\sum_{z\in\mathbb Z}N^{(2)}_{n_2}(z)=n_2$. 
It follows that
\begin{align*}
\sum_{k_1,k_2=0}^{d-1}&\varphi_{n_1,n_2}\left(\theta+\frac{2\pi}d(k_1,k_2)\right)\exp\left(-i\left\langle \theta+\frac{2\pi}d (k_1,k_2),a_n(x)\right\rangle\right)\\
&= \varphi_{n_1,n_2}(\theta)\exp\left(-i\left\langle \theta,a_n(x)\right\rangle\right)\sum_{k_1,k_2=0}^{d-1}e^{-\frac{2i\pi}d\left(\langle k,a_n(x)\rangle-\sum_{j=1}^2k_j( an_j)\right)}\\
&= d^2\indic_{d\mathbb Z^2}(a(n_1,n_2)-a_n(x))\varphi_{n_1,n_2}(\theta)\exp\left(-i\left\langle \theta,a_n(x)\right\rangle\right)\, ,
\end{align*}
where we used the fact that
\[
\sum_{k=0}^{d-1}e^{\frac{2i\pi}dk( an_j-z_j)}=d\indic_{d\mathbb Z}(an_j-z_j)\, .
\]
Hence, we have proved that
\begin{multline}\label{AAA}
\bb P\left(X''_{n_1}=\lfloor n^{3/4}x_1\rfloor,X''_{n_1+n_2}=\lfloor  n^{3/4}x_2\rfloor \right)=
\indic_{d\mathbb Z^2}(a(n_1,n_2)-a_n(x))\\
\times\frac{d^2}{(2\pi)^2}\int_{[-\pi/d,\pi/d]^2}
\varphi_{n_1,n_2}(\theta)\exp(-i\left\langle \theta,a_n(x)\right\rangle)\,d\theta\, .
\end{multline}
This identity extends~\cite[Lemma 15]{CGPPS-AOP-2014} and ends the proof of the Proposition in the case when $a(n_1,n_2)-a_n(x)\not\in d\mathbb Z^2$.\\ 
It follows from
~\cite[Propositions 18 and 19]{CGPPS-AOP-2014} that, for every $\eta\in(0,\frac18)$,
\begin{multline*}
\int_{[-\pi/d,\pi/d]^2}
\varphi_{n_1,n_2}(\theta)\exp(-i\left\langle \theta,a_n(x)\right\rangle)\,d\theta\\
=o(n^{-3/2})+
 \int_{U(\eta)}
\varphi_{n_1,n_2}(\theta) \exp(-i\left\langle \theta,a_n(x)\right\rangle)\,d\theta\, ,
\end{multline*}
where we set
\[U(\eta)=\{|\theta_i|\leq n_i^{-1/2-\eta}, i=1,2\}\, .\]

It follows from this formula combined with~\cite[Lemma 22]{CGPPS-AOP-2014} that
\begin{multline}\label{AAA0}
\int_{[-\pi/d,\pi/d]^2}
\varphi_{n_1,n_2}(\theta)\exp(-i\left\langle \theta,a_n(x)\right\rangle)\,d\theta\\
=o(n^{-3/2})+\int_{U(\eta)} \exp(-i\left\langle \theta,a_n(x)\right\rangle)   \bb E\left[\exp{\left(-\frac{1}{2}\left\langle A_{n_1,n_2} \theta,\theta\right\rangle\right)}\right]\,d\theta\, .
\end{multline}
We define, for every $\varepsilon >0$,
$$\Omega_{n_1,n_2}(\varepsilon)=\left\{ (n_1 n_2)^{-3/2}\, D_{n_1,n_2}\geq \varepsilon\right\}.$$ 
Let $v\in (0,1)$ and $\theta_0\in (0,\frac{v}{4})$ be fixed. According to~\cite[Lemma 21]{CGPPS-AOP-2014}
\[\bb P(\Omega_{n_1,n_2}^{c}(n^{-\theta_0})) = o(n^{-\frac 32})\, .\]
Thus,~\eqref{AAA0} becomes
\begin{multline}\label{AAA1}
\int_{[-\pi/d,\pi/d]^2}
\varphi_{n_1,n_2}(\theta)\exp(-i\left\langle \theta,a_n(x)\right\rangle)\,d\theta\\
=o(n^{-3/2})+ 
\quad\quad + \int_{U(\eta)}
 \exp{(-i\left\langle \theta,a_n(x)\right\rangle)} \bb E\left[\exp{(-\frac{i}{2}\left\langle A_{n_1,n_2} \theta,\theta\right\rangle}) \indic_{\Omega_{n_1,n_2}(n^{-\theta_0})} \right]\,d\theta\, .
\end{multline}
We decompose this last integral as follows
\begin{equation}\label{AAA2}
\int_{U(\eta)} \exp(-i\left\langle \theta,a_n(x)\right\rangle) \bb E\left[\exp{(-\frac{i}{2}\left\langle A_{n_1,n_2} \theta,\theta\right\rangle}) \indic_{\Omega_{n_1,n_2}(n^{-\theta_0})} \right]\,d\theta=
I_{n_1,n_2} - J_{n_1,n_2}\, ,
\end{equation}
where $I_{n_1,n_2}$ is the integral  over $\bb R^2$ and $J_{n_1,n_2}$ is the integral over $\{\exists j : |\theta_j|> n_j^{-1/2-\eta}\}.$
It follows from~\cite[proof of Lemma 23]{CGPPS-AOP-2014} that
\begin{equation}\label{AAA3}
n^{3/2} J_{n_1,n_2}\xrightarrow[n\to\infty,\frac{n_i}{n}\to t_i]{} 0.
\end{equation}
With the change of variable $u=A_{n_1,n_2}^{1/2}\theta$, $I_{n_1,n_2}$ becomes
\[I_{n_1,n_2}=\int_{\bb R^2}\bb E \left[\exp\left(-i\left\langle A_{n_1,n_2}^{-1/2}u,a_n(x)\right\rangle\right)D_{n_1,n_2}^{-1/2}\, \indic_{\Omega_{n_1,n_2}(n^{-\theta_0})}\, .\right]
e^{-\frac{1}{2}\|u\|^2}\,du\, .\]
Furthermore~\cite[Proposition~7]{CGPPS-AOP-2014}
ensures that
\[n^{-3/2}B_{n_1,n_2}\xrightarrow[n\to\infty,\frac{n_i}{n}\to t_i]{\cal L}
M_{T_1,T_2},\]
with
\[B_{n_1,n_2}=\left(\left\langle N_{n_1+\cdots +n_i},N_{n_1+\cdots +n_j}\right\rangle\right)_{i,j=1,2}=
\begin{pmatrix}
\left\langle N_{n_1},N_{n_1}\right\rangle & 
\left\langle N_{n_1},N_{n_1+n_2}\right\rangle \\
\left\langle N_{n_1},N_{n_1+n_2}\right\rangle &
\left\langle N_{n_1+n_2},N_{n_1+n_2}\right\rangle
\end{pmatrix}\]
which implies that 
\[n^{-3/2}A_{n_1,n_2}\xrightarrow[n\to\infty,\frac{n_i}{n}\to t_i]{\cal L}
\tilde M_{t_1,t_2}\, ,\]
and so that
\[n^{3/4}A^{-1/2}_{n_1,n_2}\xrightarrow[n\to\infty,\frac{n_i}{n}\to t_i]{\cal L}
\tilde M^{-1/2}_{t_1,t_2},
\quad
n^{3/2}D^{-1/2}_{n_1,n_2}\xrightarrow[n\to\infty,\frac{n_i}{n}\to t_i]{\cal L}
\tilde{D}^{-1/2}_{t_1,t_2}\, .
\]
In particular
\[\left\langle A_{n_1,n_2}^{-1/2}u,a_n(x)\right\rangle
=\left\langle n^{3/4}A_{n_1,n_2}^{-1/2}u,
\frac{a_n(x)}{n^{3/4}}\right\rangle
\xrightarrow[n\to\infty,\frac{n_i}{n}\to t_i]{\cal L}
\left\langle \tilde M^{-1/2}_{t_1,t_2}u,
a(x)\right\rangle\]
with $a(x)=(x_1, x_2-x_1)$,
and
\begin{multline*}n^{3/2}\exp\left(-i\left\langle A_{n_1,n_2}^{-1/2}u,a_n(x)\right\rangle\right)
D_{n_1,n_2}^{-1/2}  \, \indic_{\Omega_{n_1,n_2}(n^{-\theta_0})}\,\\ 
\xrightarrow[n\to\infty,\frac{n_i}{n}\to t_i]{\cal L}
\exp\left(-i\left\langle \tilde M^{-1/2}_{t_1,t_2}u,
a(x)\right\rangle \right)\tilde{D}^{-1/2}_{t_1,t_2}
\end{multline*}
It follows from~\cite[Lemma 21]{CGPPS-AOP-2014} that this sequence of random variables is bounded in ${\mathbf L}^p$ for $p>1$ and so is uniformly integrable. We infer that
\begin{multline}\label{AAA4}
n^{3/2} I_{n_1,n_2}=n^{3/2}\int_{\bb R^2}\bb E \left[\exp\left(-i\left\langle A_{n_1,n_2}^{-1/2}u,a_n(x)\right\rangle\right)
D_{n_1,n_2}^{-1/2}\, \indic_{\Omega_{n_1,n_2}(n^{-\theta_0})} \right]e^{-\frac{1}{2}\|u\|^2}\,du\\
\xrightarrow[n\to\infty,\frac{n_i}{n}\to t_i]{}
\int_{\bb R^2}\bb E \left[\exp\left(-i\left\langle \tilde M^{-1/2}_{t_1,t_2}u,
a(x)\right\rangle\right)
\tilde{D}^{-1/2}_{t_1,t_2}\right]e^{-\frac{1}{2}\|u\|^2}\,du.
\end{multline}
With the change of variable $u=\tilde M^{1/2}_{t_1,t_2}\theta$, the above limit is equal to
\begin{align}\nonumber
&\int_{\bb R^2}\exp\left(-i\left\langle \theta,
a(x)\right\rangle\right)
\bb E\left[ e^{-\frac{1}{2}\left\langle \tilde M_{t_1,t_2}\theta,\theta\right\rangle}\right]\,d\theta\\
\nonumber&=\int_{\bb R^2}\exp\left(-i\left\langle u,
x\right\rangle\right)
\bb E e^{iu_1\Delta_{T_1}+iu_2\Delta_{T_2}}\,du\\
&=4\pi^2 p_{T_1,T_2}(x_1,x_2)\label{AAA5}\, ,
\end{align}
using the change of variable $u_1=\theta_1-\theta_2$, $u_2=\theta_2$.
The proposition follows from the combination of \eqref{AAA}, \eqref{AAA1}, \eqref{AAA2}, \eqref{AAA3},~\eqref{AAA4} and~\eqref{AAA5}.
\end{proof}

\section{Conjecture and properties of the limit processes}\label{conjppt}
Considering $(\xi_\ell^{(p)})_{\ell\in\mathbb Z,p\in\mathbb N}$
a sequence of i.i.d. Rademacher random variables,
we define inductively on $p$ the processes $PA(p)$ as follows
\begin{equation}\label{PA(p)}
\forall n\in\mathbb N,\quad (PA(0))_n=n\quad\mbox{and}\quad\forall p\in\mathbb N,\  
(PA(p+1))_n:=\sum_{j=1}^{n}\xi^{(p)}_{(PA(p))_{j-1}}\, .
\end{equation}
The diffusivity of the process $PA(p)$ is measured by 
its normalizing exponent $\alpha_p$. 
The greater the exponent $\alpha_p$ is, the more diffusive  $PA(p)$ is.
Proposition~\ref{PRO1} (applied with $\beta=2$) enlights the following recurrence relation
between the successive normalizing exponents of the $PA(p)$ for
$p\in\{2,3\}$~:
\begin{equation}\label{recalphap}
\alpha_p=1-\frac{\alpha_{p-1}}2=f(\alpha_{p-1}),\quad \mbox{with }\alpha_1=1/2\, .
\end{equation}
Observe that the function $f$ given by $f(x)=1-\frac x2$ is
a contracting decreasing affine function with
a single fixed point $x_0=\frac 23$ and note that $\alpha_1=1/2<x_0$. Thus
\[
\alpha_1=\frac 12<\alpha_3=\frac 58<x_0<\alpha_2=\frac 34\ .
\]
To understand the relation~\eqref{recalphap} between $\alpha_p$ and $\alpha_{p-1}$, a  key remark made by Kesten and Spitzer in~\cite{KS-1979} is given by the three following lines of formulas leading to an asymptotic estimate of the variance of $(PA(p))_n$. 
Using the definition of $PA(p)$ combined with the fact that $(\xi_y^{(p)})_y$ is a sequence of centered independent random variables independent of $PA(p-1)$, we first notice that 
\begin{align*}
\mathbb E\left[(PA(p))_n^2\right]&=\mathbb E\left[\left(\sum_{k=1}^n\xi^{ (p)}_{(PA(p-1))_k}\right)^2\right]=\mathbb E\left[\left(\xi_0^{(p)}\right)^2\right]
\mathbb E[V_n]\, ,
\end{align*} 
where $V_n:=\sum_{k=l=1}^n\mathbf 1_{\{(PA(p-1))_k=(PA(p-1))_l\}}$ is the self-intersection local time of $PA(p-1)$ up to time $n$. 
Second, using the stationarity of the increments of $PA(p-1)$, it follows that
\begin{align*}
\mathbb E[V_n]
&=\sum_{k,l=1}^n\mathbb P\left((PA(p-1))_k=(PA(p-1))_l\right)
=\sum_{k,l=1}^n\mathbb P\left((PA(p-1))_{|l-k|}=0\right)\, .
\end{align*}
Provided a local limit theorem holds for $PA(p-1)$ (which was proved only when $p\in\{2,3\}$), this leads to
\[
\mathbb E\left[(PA(p))_n^2\right]
\sim c \sum_{k,l=1}^n (1+|l-k|)^{-\alpha_{p-1}}\sim c' n^{2-\alpha_{p-1}}=c'n^{2\alpha_p}\, ,\]
leading to the recurrence relation~\eqref{recalphap}.
From a probabilistic point of view, the reason why the function $f$ appearing in the recurrence relation~\eqref{recalphap} is decreasing is that, roughly speaking, the more diffusive the process $PA(p-1)$ is, the smaller is its self-intersection local time $V_n$
(and also the smaller is the probability that the process visits the same site at two different times), and so if $PA(p)$ is more diffusive than $PA(p-1)$, then $PA(p+1)$ will be less diffusive than $PA(p)$.\\
In light of the results established in the present article, we formulate the following conjecture.
\begin{conj}[Higher order Iterated PAPA]
It should be possible to apply the argument of Proposition~\ref{PRO1} inductively in $p$ for
$PA(p)$ with $\beta=2$. This leads us to conjecture that the normalizing exponent $\alpha_p$ of $PA(p)$
should be given by
\begin{equation}\label{alphap}
\alpha_0=1\quad\mbox{and}\quad
\alpha_p:=1-\frac{\alpha_{p-1}}2\, ,
\end{equation}
i.e. by
$\alpha_p=\frac 23+\frac{(-1/2)^{p}}{3}$. 
The sequence of odd (resp. even) exponents $\alpha_p$ is increasing (resp. decreasing) in $p$ and converges to $\frac23$.\\
Let us denote for any $p\geq 1$,
\[\left(\mathcal P_n^{(p)}:=\left(n^{-\alpha_p}(PA(p))_{\lfloor nt\rfloor}\right)_{t\ge 0}\right)_{n\in\mathbb N^*}\] 
We also conjecture that, for any integer $p\ge 4$, the family of processes \[\left(\mathcal P_n^{(1)},\ldots,\mathcal P_n^{(p)} \right)_{n\in\mathbb N^*}\] 
converges in distribution to the process 
$(\Xi^{(1)},\ldots,\Xi^{(p)})$ where $\Xi^{(p)}$ is defined inductively by
\begin{equation}\label{Zp}
\Xi^{(0)}=id_{[0,+\infty)}\quad\mbox{and}\quad \Xi^{(p+1)}_t:=\int_{\mathbb R}L^{(p)}
(t,x)\, {\rm d}W^{(p)}(x)\, ,
\end{equation}
where $L^{(p)}$ is the local time of the process
$\Xi^{(p)}$ and where $(W^{(p)})_{p\ge 0}$ is a sequence of independent standard Brownian motions.\\
As pointed out by a referee, a natural question is whether the sequence of processes $(\Xi^{(p)})_p$ converges to a process $\Xi^{(\infty)}$ (with local time $(L^{(\infty)} (t,x))_{t,x}$) that would be a fixed point in distribution of the induction relation~\eqref{Zp}, i.e. that would satisfy the following identity in distribution
\[
\left(\Xi^{(\infty)}_t\right)_t\stackrel{distrib}{=} \left(\int_{\mathbb R}L^{(\infty)} (t,x)\, {\rm d}W^{(1)}(x)\right)_t \, ,
\]
and this also leads to the question of the existence of 
such a process $\Xi^{(\infty)}$.
\end{conj}
A main difficulty to prove this conjecture for $PA(p)$ (with $p>3$) is to check 
 the assumptions (ii), (iii) and (iv) of Proposition \ref{PRO1} related to the Local Limit Theorem (and multi-time versions of it) of $PA(p-1)$. Even the Local Limit Theorem for PA(2) was hard to prove. The proof of the Local Limit Theorem for PA(2) strongly uses the fact that PA(1) has independent increments. This question becomes even more intricate for PA(3) because of the dependence of the increments of PA(2) (this dependence does not disappear asymptotically).

Let us observe that the limit processes $\Xi^{(p)}$ defined in~\eqref{Zp} are the analogue, in continuous time, of the 
$PA(p)$ defined in~\eqref{PA(p)}. 
Indeed, the last equation of~\eqref{PA(p)} can be rewritten as follows 
\[
(PA(p+1))_n=\sum_{\ell\in\mathbb Z} \xi^{(p)}_\ell N_n^{(p)}(\ell)
\, ,
\]
where $N_n^{(p)}$ stands for the local time of the $PA(p)$, given by
\[
N_n^{(p)}(\ell):=\#\left\{k=0,...,n-1\, :\, (PA(p))_k=\ell\right\}\, .
\]
We will now see that the family of processes
$(\Xi^{(p)})_p$ is well defined inductively and study some of their properties.
\begin{thm}\label{THH}
For any integer $p\ge 0$, 
the process $\Xi^{(p)}$:
\begin{itemize}
    \item is well defined,
    \item is $\alpha_p$-self-similar (i.e. for every $c>0$, $(\Xi^{(p)}(ct))_{t\geq 0}$ has the same distribution as $(c^{\alpha_p} \Xi^{(p)}(t))_{t\geq 0}$) with $\alpha_p$ given by~\eqref{alphap}, 
    \item has stationary increments and admits a local time $L^{(p)}$ such that, 
    \begin{equation}\label{sqInt}
    \forall t>0,\quad \mathbb E\left[\int_{\mathbb R}\left(L^{(p)}(t,x)\right)^2\, {\rm d}x\right]<\infty\, .\end{equation}
\end{itemize}
Furthermore, for any $p\geq 2$, the process $\Xi^{(p)}$ has a continuous modification which, with probability one, is locally $\gamma$-Hölder continuous for all $0<\gamma<\alpha_p - \frac12$ and
$$\mathbb E\left[ (\Xi_1^{(p)} )^2\right] =  \mathbb E\left[ \int_{\bb R} (L^{(p-1)}(1,x))^2\, {\rm d}x \right]\, .$$
For $p=3,$
\begin{equation}\label{YYY}
\mathbb E\left[ (\Xi_1^{(3)} )^2\right] = \frac{16}{5}\sqrt{\frac{2}{\pi}}  {\mathbb E}\left[ \frac{1}{\sqrt{V_1^{(B)}}}\right]\, ,
\end{equation}
 where 
 $V_1^{(B)}= \int_{\bb R} (L^{(B)}(1,x))^2\, {\rm d}x$ 
 denotes the self-intersection local time of the real Brownian motion.
\end{thm}
\begin{rem}
An extension of the above formula~\eqref{YYY} to any $p\geq 4$ could be obtained under the assumption that the random vectors $(\Xi_s^{(p-1)},\Xi_t^{(p-1)}), 0<s<t$ admit a density and it should then be given by 
$$\mathbb E\left[ (\Xi_1^{(p)} )^2\right] =  C_p\,  \mathbb E\left[ \left(\int_{\bb R} (L^{(p-2)}(1,x))^2\, {\rm d}x \right)^{-1/2}\right]\, ,$$
where $C_p= \frac{1}{\sqrt{2\pi}} \frac{1}{\alpha_p (1-\alpha_{p-1})}.$
\end{rem}
\begin{proof}[Proof of Theorem~\ref{THH}]
We proceed by induction. We know that the statement holds true for 
$p=0$ (with $L^{(0)}(t,x)=\mathbf 1_{[0,t)}(x)$) and for $p=1$ (recall that $\Xi^{(1)}$ is a Brownian motion). 

Let $p\ge 2$. Let us assume that the statement holds true for $p-1$ and $p-2$ and let us prove it holds true for $p$. To simplify notations, we write $\alpha:=\alpha_{p-1}$, $L:=L^{(p-1)}$ and $W:=W^{(p-1)}$. 

The fact that $L$ is a.s. square integrable in space and independent of $W$ ensures that conditionally to $\Xi^{(p-1)}$, the process $\Xi^{(p)}$ defines a Gaussian process. 

Let us prove that $\Xi^{(p)}$ is $(1-\frac \alpha 2)-$self-similar. To this end, we fix $c>0$.
Since $\Xi^{(p-1)}$ is $\alpha$-self-similar, its local time $L$ satisfies the following self-similarity condition 
$$(L(ct,x))_{t,x} \overset{(d)}{=} c^{1-\alpha} (L(t,x c^{-\alpha}))_{t,x}\, ,$$
(see e.g.~\cite[Proposition 10.4.8]{Samorodnitsky}).

Therefore
\begin{align*}
\left(\Xi^{(p)}_{ct}\right)_t&:=\left(\int_{\mathbb R}L(ct,x)\, {\rm d}W(x)\right)_t\\
&\overset{(d)}{=} \left(c^{1-\alpha}\int_{\mathbb R} L(t,x/c^\alpha)\, {\rm d}W(x)\right)_t\\
&\overset{(d)}{=} \left(c^{1-\frac{\alpha}2}\int_{\mathbb R} L(t,y)\, {\rm d}W(y)\right)_t\, ,
\end{align*}
where we used the fact that $L$ is independent of $W$ and that $(W(c^\alpha y))_y\overset{(d)}{=}c^{\frac \alpha 2}W$. Thus $\Xi^{(p)}$ is $\left(1-\frac\alpha 2\right)$-self-similar, as announced.\\
The fact that $\Xi^{(p)}$ has stationary increments follows from the following fact. Let $t>0$. We observe that 
\begin{align*}
    \left(\Xi^{(p)}(t+s)-\Xi^{(p)}(t)\right)_s&=\left(\int_{\bb R} \left(L(t+s,x)-L(t,x)\right)\, {\rm d}W(x)\right)_s\\
    &= \left(\int_{\bb R} L'(s,x-\Xi^{(p-1)}_t)\, {\rm d}W(x)\right)_s
\end{align*}
where we set $L'$ for the local time of $\Theta^{(t)}=\left(\Xi^{(p-1)}_{t+s}-\Xi^{(p-1)}_t\right)_s$. 
Since $W$ has stationary increments and is independent of $\Xi^{(p-1)}$, it follows that
\begin{align*}
    \left(\Xi^{(p)}(t+s)-\Xi^{(p)}(t)\right)_s
    &\overset{(d)}{=} \left(\int_{\bb R} L'(s,x)\, {\rm d}W(x)\right)_s\, .
\end{align*}
And, since $\Xi^{(p-1)}$ has stationary increments, we know that $\Theta^{(t)}\overset{(d)}{=}\Xi^{(p-1)}$ and so
\begin{align*}
    \left(\Xi^{(p)}(t+s)-\Xi^{(p)}(t)\right)_s
    &\overset{(d)}{=} \left(\int_{\bb R} L(s,x)\, {\rm d}W(x)\right)_s=\Xi^{(p)}\, .
\end{align*}
This ends the proof of the stationarity of $\Xi^{(p)}$.\\
To prove~\eqref{sqInt}, we 
apply~\cite[Theorem~(21.9)]{GH} to the process~$\Xi^{(p)}$. Since \[\Xi^{(p)}_{t}-\Xi^{(p)}_s\overset{(d)}{=}\Xi^{(p)}_{t-s}\overset{(d)}{=}(t-s)^{1-\frac\alpha 2}\Xi^{(p)}_1\, , \] 
Condition~\cite[(21.10)]{GH} will follow from the finiteness of the following quantity
\begin{align}
\int_{[0,T]^2\times\mathbb R}& \mathbb E\left[e^{i\theta |t-s|^{1-\frac\alpha 2}\Xi_1^{(p)}}\right]{\rm d}s{\rm d}t{\rm d}\theta\\
&= \int_{[0,T]^2}|t-s|^{\frac\alpha 2-1}\, {\rm d}t{\rm d}s\int_{\mathbb R} \mathbb E\left[e^{iu\Xi_1^{(p)}}\right]\,{\rm d}u\, ,
\end{align}
and so from the finiteness of
\[
D:=\int_{\mathbb R} \mathbb E\left[e^{iu\Xi_1^{(p)}}\right]\, {\rm d}u\, .
\]
But, conditioning with respect to $\Xi^{(p-1)}$, $\Xi_1^{(p)}$ has a centered Gaussian distribution with variance $V_1=\int_{\mathbb R} (L(1,x))^2\, {\rm d}x$ the self-intersection local time of the process $\Xi^{(p-1)}$.
Thus
\begin{equation}\label{domD}
D=\int_{\mathbb R}\mathbb E\left[\exp\left(-\frac{u^2 V_1}{2}\right)\right]\, {\rm d}u 
=\sqrt{2\pi} \mathbb E\left[ \frac{1}{\sqrt{V_1}}\right]\leq  2\sqrt{\pi} \mathbb E\left[\sqrt{\Xi^{(p-1,*)}}\right]\, ,
\end{equation}
where we set $\Xi^{(p-1,*)}:= \sup_{s\in [0,1]} \left| \Xi_s^{(p-1)} \right|$ and used the following fact coming from the Cauchy-Schwarz inequality
\begin{align*}
1&=\int_{\mathbb R}L(1,x)\, {\rm d}x=\int_{-\Xi^{(p-1,*)}}^{\Xi^{(p-1,*)}}L(1,x)\, {\rm d}x\\
&\le \sqrt{\int_{-\Xi^{(p-1,*)}}^{\Xi^{(p-1,*)}}\, {\rm d}x}\sqrt{\int_{-\Xi^{(p-1,*)}}^{\Xi^{(p-1,*)}}(L(1,x))^2}\, {\rm d}x=\sqrt{2\Xi^{(p-1,*)}}\sqrt{V_1} \, .
\end{align*}
Using the fact that, conditionally to $\Xi^{(p-2)}$, $\Xi^{(p-1)}$ is centered gaussian, combined e.g. with Borell's inequality, we obtain that
\begin{eqnarray*}
\mathbb E\left[\Xi^{(p-1,*)}\right] &=&\int_0^{\infty} \mathbb P\left( \sup_{s\in [0,1]} | \Xi^{(p-1)}_s | \geq x\right) dx \\
&\leq & \int_0^{\infty} 2\mathbb E\left[ e^{-\frac{x^2}{2V_1^{(p-2)}}}\right]\, {\rm d}x\\
&\leq &   \mathbb E\left[\sqrt{2\pi V_1^{(p-2)}}\right]<+\infty\, ,
\end{eqnarray*}
where we set $V_1^{(p-2)}:=\int_{\mathbb R}(L^{(p-2)}(1,x))^2\, {\rm d}x$ and used~\eqref{sqInt} at range  $p-2$.\\
This implies that $\sqrt{\Xi^{(p-1,*)}}\in {\mathbb L}^2\subset {\mathbb L}^1$. The finiteness of $D$ follows via~\eqref{domD}. As announced, this allows us to apply~\cite[Theorem~(21.9)]{GH} and to conclude~\eqref{sqInt}.

Let $p\geq 2$ and $0\leq s <t $. Then, using self-similarity and stationarity of the increments of $\Xi^{(p)}$, we have 
\[
\bb E[|\Xi^{(p)}(t) - \Xi^{(p)}(s)|^2] = \bb E[\Xi^{(p)}(t-s)^2]\\
= |t-s|^{2\alpha_p} \bb E[\Xi^{(p)}(1)^2]\, ,
\]
where 
\begin{align*}
\bb E\left[\left(\Xi^{(p)}(1)\right)^2\right]&=\mathbb E[V_1]=\int_\bb R \bb E[(L(1,x))^2]\,{\rm d}x  <+\infty\, .
\end{align*}
The Hölder continuity directly follows from classical Kolmogorov's Theorem. 

Let $p=3$. Remark that $\Xi^{(3)}$ is the process $\Gamma$ obtained in Theorem~\ref{THM} or in Theorem \ref{THMtriple} when both random sceneries are given by sequences of i.i.d. Rademacher random variables. 
Using formula (5) in \cite[Theorem 3]{CGPPS-AOP-2014},
\[\begin{aligned}
\bb E[\Xi^{(3)} (1)^2]&=
\int_\bb R \bb E[(L^{(\Delta)}(1,x))^2]\,{\rm dx}\\
&=\int_{[0,1]^2} \int_\bb R p_{\Delta_u,\Delta_v}(x,x) \, {\rm dx}\, {\rm du} {\rm dv}\, ,
\end{aligned}\]
with 
$p_{\Delta_u,\Delta_v}$
        the continuous density of $(\Delta_u, \Delta_v).$
As the density of the random vector $(\Delta_v-\Delta_u,\Delta_v)$
is equal to $p_{\Delta_u,\Delta_v}(y-x,y)$ we can rewrite
\[\int_\bb R p_{\Delta_u,\Delta_v}(x,x) \,{\rm dx}
=\int_\bb R p_{\Delta_v-\Delta_u,\Delta_v}(0,x) \, {\rm dx}
=: p_{\Delta_v-\Delta_u}(0)\, .\]
Then,
\[\begin{aligned}
\bb E[\Xi^{(3)}(1)^2]&=\int_{[0,1]^2} p_{\Delta_v-\Delta_u}(0) \, {\rm du} {\rm dv}
=2\int_{0<u<v<1} p_{\Delta_v-\Delta_u}(0) \, {\rm du} {\rm dv}\\
&=2\int_{0<u<v<1} p_{\Delta_{v-u}}(0) \, {\rm du} {\rm dv}.
\end{aligned}\]
From Corollary 2 in \cite{CGPPS-AOP-2014} and the scaling property of the Brownian motion's local time, we get 
\begin{align*}
p_{\Delta_{v-u}}(0)&=(2\pi)^{-\frac 12}\, 
    {\mathbb E}\left[ \left(\int_{\bb R} (L^{(B)}(v-u,x))^2\, {\rm d}x \right)^{-1/2}\right]\\
    &=(v-u)^{-3/4} (2\pi)^{-\frac 12}\, 
    {\mathbb E}\left[ \left(\int_{\bb R} (L^{(B)}(1,x))^2\, {\rm d}x \right)^{-1/2}\right]\, .
\end{align*}
After computations, we obtain
 $$\bb E[\Xi^{(3)}(1)^2] = \frac{16}{5}\sqrt{\frac{2}{\pi}}  {\mathbb E}\left[ \frac{1}{\sqrt{V^{(B)}_1}}\right]\, ,$$
 where 
 $V^{(B)}_1= \int_{\bb R} (L^{(B)}(1,x))^2\, {\rm d}x$ denotes the self-intersection local time of the real Brownian motion.
\end{proof}

\section{Other three-dimensional randomly oriented models}\label{othermodels}
In our model, the second coordinate is driven by the first one which itself is driven by the third one which is a random walk.\\
It can appear natural to consider other three-dimensional randomly oriented models.\\
Let us first present such models in which one or two coordinates are directly driven by the other ones which will appear to be a one- or two-dimensional random walk.
\begin{enumerate}
\item If, in our model, the orientations $(0,\xi_{x_0}^{(2)},0)$ of the lines of the form $\{(x_0,y,z_0); y\in\mathbb R\}$ are replaced by 
$(0,\xi_{z_0}^{(2)},0)$, then the two first coordinates $(X_n,Y_n)$ correspond to a two-dimensional PAPA driven by the random walk $(Z_n)_n$, that is a couple of PAPAs with independent sceneries but driven by the same random walk $(Z_n)_n$.
\item 
If, we consider now a model in which only the lines parallel to $(1,0,0)$ are oriented, then
the two last coordinates $(Y_n,Z_n)_n$ perform a two-dimensional random walk and the first coordinate $(X_n)_n$ is a PAPA; more precisely~:
\begin{enumerate}
\item If the lines of the form $\{(x,y_0,z_0); x\in\bb Z\}$ have orientation $(\xi_{z_0}^{(1)},0,0)$, then the first coordinate is a PAPA driven by the random walk $(Z_n)_n$.
\item If the lines of the form $\{(x,y_0,z_0); x\in\bb Z\}$ have orientation $(\xi_{y_0}^{(1)},0,0)$, then the first coordinate is a PAPA driven by the random walk $(Y_n)_n$.
\item If the lines of the form $\{(x,y_0,z_0); x\in\bb Z\}$ have orientation $(\xi_{(y_0,z_0)}^{(1)},0,0)$, then the first coordinate is a PAPA driven by the two-dimentional random walk $(Y_n,Z_n)_n$, as studied 
by Bolthausen in~\cite{Bolthausen}.
\end{enumerate}
\end{enumerate}

More intricate 
three-dimensional randomly oriented models  appear when "circular dependence appear":
\begin{enumerate}
\item[(3)] We can e.g. consider a model in which the random orientations of the lines of the form $\{(x,y_0,z_0); x\in\bb Z\}$ 
depend on $y_0$, while the orientations of the random orientations of the lines
$\{(x_0,y,z_0); y\in\bb Z\}$ depend on $x_0$, then the motion of $(X_n)_n$ is driven by the motion of $(Y_n)_n$
and conversely the motion of $(Y_n)_n$ is driven by
the motion of $(X_n)_n$. The study of such model becomes much harder. Lots of questions are open even for the two-dimensional version, that is the random walk  in the randomly oriented Manhattan lattice. This model was introduced in~\cite[Section 5]{GPLN1}, see also the survey~\cite[Section 3.5]{Pene} for  a presentation of this model and a conjecture and~\cite{LTV} for the only existing estimates for this model.
\end{enumerate}

{\bf Acknowledgement.} F. P. and F. W. conducted this work  within the framework of the Henri Lebesgue Center (ANR-11-LABX-0020-01) and 
 with the support of
 the ANR project RAWABRANCH (ANR-23-CE40-0008). 
 F. W. thanks the IRMAR, CNRS UMR 6625, University of Rennes 1, for its hospitality.

\end{document}